\documentclass[leqno,11pt]{article}
\usepackage[colorlinks,pagebackref,hypertexnames=false]{hyperref}

\usepackage{amsmath,amsthm,amssymb,mathrsfs} 
\usepackage[alphabetic,backrefs]{amsrefs}
\usepackage[T1]{fontenc}
\usepackage{ae,aecompl}
\usepackage{times}
\usepackage{microtype}

\usepackage[english]{babel}

\usepackage{bookmark}

\numberwithin{equation}{section}

\newcommand{\U}{{\rm U}}

\newcommand{\Ad}{\mathrm{Ad}}

\renewcommand{\epsilon}{\varepsilon}

\newcommand{\diag}{\mathrm{diag}}

\newcommand{\dvol}{\mathop\mathrm{dvol}\nolimits}
\newcommand{\id}{\mathrm{id}}

\DeclareMathOperator\GG{G}

\DeclareMathOperator\SU{SU}

\def\<{\mathopen{}\left<}
\def\>{\right>\mathclose{}}
\def\({\mathopen{}\left(}
\def\){\right)\mathclose{}}

\newtheorem{theorem}{Theorem}

\newtheorem{lemma}{Lemma}

\newtheorem{proposition}{Proposition}
\newtheorem{remark}{Remark}

\numberwithin{equation}{section}

\author{Jason D. Lotay \\ University College London \and  Goncalo Oliveira \\ Universidade Federal Fluminese
}

\title{$\GG_2$-instantons on noncompact $\GG_2$-manifolds: results and open problems} 

\date{}

\begin{document}
\maketitle


 \begin{abstract}
\noindent We survey the known existence and non-existence results for $\GG_2$-instantons on non-compact cohomogeneity-1 $\GG_2$-manifolds and their consequences, including an explicit example of a family of $\GG_2$-instantons where bubbling, removable singularities and conservation of energy phenomena occur.  We also describe several open problems for future research.
\end{abstract}

\tableofcontents

\section{Introduction}

A $\GG_2$-instanton is a special kind of Yang--Mills connection on a Riemannian 7-manifold with holonomy group contained in $\GG_2$ (a so-called $\GG_2$-manifold).  
One can think of $\GG_2$-instantons as analogues of anti-self-dual connections in $4$ dimensions.  This analogy motivates the hope of using $\GG_2$-instantons to construct enumerative invariants of $\GG_2$-manifolds.  In this review article we shall be focusing on $\GG_2$-manifolds and $\GG_2$-instantons constructed using symmetry techniques. It is important to note that, using the fact that $\GG_2$-manifolds are Ricci flat, one sees that holonomy $\GG_2$-manifolds\footnote{Those $\GG_2$-manifolds whose holonomy is exactly $\GG_2$ will be referred to as holonomy $\GG_2$-manifolds.} admitting continuous symmetries must be noncompact.  Symmetry techniques thus have a somewhat limited scope of applicability, but they do lead to simplifications which make hard problems in the field tractable in this special setting, giving in several cases explicit non-trivial examples as well as significant results which may be useful in the general theory.   Here, we shall summarize the known existence and non-existence results for $\GG_2$-instantons in the symmetric setting, as well as their consequences. For example, we shall see an explicit example of a family of $\GG_2$-instantons for which bubbling and removable singularities phenomena happen.  We shall also describe several important open problems for future research.

\subsection*{Background}

Let $(X^7,\varphi)$ be a $\GG_2$-manifold\footnote{For further background on $\GG_2$-manifolds, the reader may wish to consult Joyce's book \cite{Joyce2000}.}, which implies that the 7-manifold $X^7$ is endowed with a $3$-form $\varphi$ which is closed and determines a Riemannian metric $g$ with respect to which $\varphi$ is also coclosed. We shall denote $\ast \varphi$ by $\psi$ for convenience.  Let $P \rightarrow X$ be a principal bundle with structure group $\GG$ which we suppose to be a  compact and semisimple Lie group. A connection $A$ on $P$ is said to be a $\GG_2$-instanton if
\begin{equation}\label{eq:G2Instanton}
F_A \wedge \psi = 0.
\end{equation}
Equivalently, $\GG_2$-instantons satisfy the following $\GG_2$-analogue of the ``anti-self-dual''  condition:
\begin{equation}\label{eq:ASD}
F_A\wedge\varphi=-*F_A.
\end{equation}
As far as the authors are aware, the first time $\GG_2$-instantons appeared in the literature was in \cite{Corrigan1983}. This reference investigates generalizations of the anti-self-dual gauge equations, in dimension greater than $4$, and $\GG_2$-instantons appear there as an example.\\
More recently, the study of $\GG_2$-instantons has gained a special interest,  primarily due to Donaldson--Thomas' suggestion  \cite{Donaldson1998} that it may be possible to use $\GG_2$-instantons to define invariants for $\GG_2$-manifolds, inspired by Donaldson's pioneering work on anti-self-dual connections on  4-manifolds. Later Donaldson--Segal \cite{Donaldson2009}, Haydys \cite{Haydys2011}, and Haydys--Walpuski \cite{Haydys2015} gave further insights regarding this possibility.\\
On a compact holonomy $\GG_2$-manifold $(X^7, \varphi)$ any harmonic $2$-form is ``anti-self-dual'' as in \eqref{eq:ASD}, hence  any complex line bundle $L$ on $X$ admits a $\GG_2$-instanton, namely that whose curvature is the harmonic representative of $c_1(L)$.  However, the construction of non-abelian $\GG_2$-instantons on compact $\GG_2$-manifolds is much more involved. In the compact case, the first such examples were constructed by Walpuski  \cite{Walpuski2011}, over Joyce's $\GG_2$-manifolds (see \cite{Joyce2000}).  S\'a Earp and Walpuski's work  \cites{SaEarp2015 ,Walpuski2015} gives an abstract construction of $\GG_2$-instantons, and currently one example, on the other known class of compact $\GG_2$-manifolds, namely ``twisted connected sums'' (see \cites{Kovalev2003,Haskins}). More recently, M\'enet--S\'a Earp--Nordstr\"om constructed other examples of $\GG_2$-instantons on twisted connected sum $\GG_2$-manifolds \cite{Menet2015}.\\
On complete, noncompact, holonomy $\GG_2$-manifolds, the first examples of $\GG_2$-instantons where found by Clarke in \cite{Clarke14}, and further examples were given by the second author in \cite{Oliveira2014} and by both authors in \cite{Lotay2018}.  We shall describe these examples in this article, and discuss natural open problems which arise from their study.

\subsubsection*{Acknowledgments}

We thank Spiro Karigiannis and Ragini Singhal for their comments and suggestions on a previous version of this manuscript.

\section{Preliminaries}

In this section we shall be considering manifolds that (in a dense open set) can be written as $X^7=I_t
\times M^6$ with $I_t \subset \mathbb{R}$ an interval with coordinate $t\in\mathbb{R}$. Then, we will write the $\GG_2$-instanton conditions as evolution equations in the $t$ coordinate and make some observations about these equations.

\subsection*{Evolution equations}

\noindent Before turning to $\GG_2$-instantons, we recall here how to write the equations for a torsion-free $\GG_2$-structure on $X $ as evolution equations.  This requires the notion of an $\SU(3)$-structure on an almost complex $6$-manifold $(M,J)$, which consists of a pair $(\omega, \gamma_2)$ of a real $(1,1)$-form and a real $3$-form respectively, such that
$$\omega \wedge \gamma_2 =0 , \quad \omega^3 =  \frac{3}{2} \gamma_1 \wedge \gamma_2 , $$
where $\gamma_1=-J\gamma_2$. Now  let $(\omega(t), \gamma_2(t))$ be   a $1$-parameter family of $\SU(3)$-structures, parametrized by the coordinate $t \in I_t$, and consider the $\GG_2$-structure on $X$ given by
\begin{equation}\label{eq:G2str}
\varphi = dt \wedge \omega(t) + \gamma_1(t) , \quad \psi= \frac{\omega^2(t)}{2} - dt \wedge \gamma_2(t).
\end{equation}
The equations $d\varphi=0$ and $d\psi=0$, for the $\GG_2$-structure to be torsion-free, turn into the following evolution equations for the $\SU(3)$-structures $(\omega(t), \gamma_2(t))$:
\begin{equation}\label{eq:Hitchinflow}
\dot{\gamma}_1= d\omega , \quad  \omega \wedge \dot{\omega}
= -d\gamma_2,
\end{equation}
subject to the constraints $d \gamma_1=0=d\omega^2$ for all $t$. These evolution equations are the so-called ``Hitchin flow''\footnote{The nomenclature ``Hitchin flow'' may be misleading. Indeed, the system \eqref{eq:Hitchinflow} is not parabolic and does not satisfy the usual regularity properties of geometric flows \cite{Bryant2010}.} and the constraint $d \gamma_1=0=d\omega^2$ is usually called the half-flat\footnote{The name ``half-flat'' comes from the fact that the condition implies the vanishing of exactly half of the torsion components of $(\omega, \gamma_2)$ as an $\SU(3)$-structure.} condition. In fact, this constraint is compatible with the Hitchin flow \eqref{eq:Hitchinflow}, meaning that if one imposes the half-flat condition on the $\SU(3)$-structure at some $t_0 \in I_t$, the evolution equations \eqref{eq:Hitchinflow} will preserve it for all $t \in I_t$. See \cite{Madsen2013} for more on half-flat $\SU(3)$-structures in a case relevant to some of the works reviewed in this article.\\
The $\GG_2$-structure $\varphi$ on $X$ obtained from solving the Hitchin flow induces the metric $g=dt^2 + g_t$, where $g_t$ is the metric on $\lbrace t \rbrace \times M$ compatible with the $\SU(3)$-structure $(\omega(t), \gamma_2(t))$.  For example, if we take $(\omega,\gamma_2)$ to be nearly K\"ahler on $M$, i.e.
\begin{equation*}
d\omega=3\gamma_1,\quad d\gamma_2=-2\omega^2,
\end{equation*}
and $g_M$ is the nearly K\"ahler metric on $M$, then the $\GG_2$-structure $\varphi$ given by solving \eqref{eq:Hitchinflow} is
\begin{equation}\label{eq:G2cone}
\varphi=t^2dt\wedge\omega+t^3\gamma_{1},\quad \psi=t^4\omega^2/2-t^3dt\wedge\gamma_{2},
\end{equation}
which gives a conical metric $g=dt^2+t^2g_M$ on $X$.
\\
Now let us consider a principal $\GG$-bundle $P$ on $X$ pulled back from $M$. There is no loss of generality in assuming this, as well as in working in temporal gauge, i.e. in setting the connection on $P$ over $X$ to be of the form $A=a(t)$, where $a(t)$ is a $1$-parameter family of connections on $P$, now seen as a vector bundle over $M$. The curvature of $A$ is given by $F_A = dt \wedge \dot{a} + F_a(t)$, where $F_a(t)$ is the curvature of $a(t)$ as a connection on $P$ over $M$. Then, the $\GG_2$-instanton equation \eqref{eq:G2Instanton} for $A$, turns into the following evolution equation for $a(t)$:
\begin{equation}\label{eq:evolution}
\dot{a} \wedge \frac{\omega^2}{2} - F_a \wedge \gamma_2 =0, \quad F_a \wedge \frac{\omega^2}{2} = 0.
\end{equation}
Applying $\ast_t$, the Hodge-$\ast$ of the metric $g_t$, to both sides of \eqref{eq:evolution} we have
\begin{align}
\label{eq:MEvolution1}
J_t \dot{a} & =  - \ast_t \left( F_a \wedge \gamma_2 \right), \\ \label{eq:MEvolution2}
\Lambda_t F_a & =  0,
\end{align}
with $\Lambda_t$ denoting the metric dual of the operation of wedging with $\omega(t)$. 
As   for the Hitchin flow, the evolution equation \eqref{eq:MEvolution1} is compatible with the constraint \eqref{eq:MEvolution2}. The discussion above and this claim can be formally stated as follows.

\begin{lemma}\label{lem:Constraint}
Let $X=I_t\times M$ be equipped with a $\GG_2$-structure $\varphi$ as in \eqref{eq:G2str} satisfying $\omega \wedge d \omega =0$ and $\omega \wedge \dot{\omega}=- d \gamma_2$, which is equivalent to $d\psi=0$. Then, $\GG_2$-instantons $A$ for $\varphi$ are in one-to-one correspondence with $1$-parameter families of connections $\lbrace a(t) \rbrace_{t \in I_t}$ solving the evolution equation
\begin{equation}\label{eq:IEvolution}
J_t \dot{a} =  - \ast_t \left( F_a \wedge \gamma_2 \right),
\end{equation}
subject to the constraint $\Lambda_t F_a =0$. Moreover, this constraint is compatible with the evolution: more precisely, if it holds for some $t_0 \in I_t$, then it holds for all $t \in I_t$.
\end{lemma}
\begin{proof}
Both the evolution equation and constraint follow immediately from the previous discussion, more precisely equations \eqref{eq:MEvolution1} and \eqref{eq:MEvolution2}. The proof that the constraint is preserved by the evolution follows from computing
\begin{align*}
\frac{d}{dt} \left( F_a \wedge  \omega^2 \right) & = d_a \dot{a} \wedge \omega^2 + F_a \wedge \frac{d}{dt}\omega^2 =  d_a (\dot{a} \wedge \omega^2) -2 F_a \wedge d \gamma_2\\
& = 2 d_a(F_a \wedge \gamma_2) - 2 F_a \wedge d \gamma_2 =  0,
\end{align*}
where we used \eqref{eq:Hitchinflow}, \eqref{eq:evolution},  
\eqref{eq:IEvolution} and the Bianchi identity $d_a F_a=0$.
\end{proof}

\begin{proposition}\label{prop:Compatible_Constraint}
In the setting of Lemma \ref{lem:Constraint}, suppose that the family of $\SU(3)$-structures $(\omega(t), \gamma_2(t))$ depends real analytically on $t$, and let $a(0)$ be a real analytic connection on $P$ such that $\Lambda_0 F_a(0)=0$. Then there is $\epsilon>0$ and a $\GG_2$-instanton $A$ on $(-\epsilon , \epsilon) \times M^6$ with $A\vert_{\lbrace 0 \rbrace \times M^6}=a(0)$.
\end{proposition}
\begin{proof}
This is immediate from applying the Cauchy-Kovalevskaya theorem to \eqref{eq:IEvolution}.
\end{proof}

\begin{remark}
We can similarly derive evolution equations defining $\GG_2$-monopoles, i.e.~pairs $(A, \Phi)$ where $A$ is a connection on $P$ and $\Phi$ is a section of the adjoint bundle, $\mathfrak{g}_P$, satisfying
$$\ast \nabla_A \Phi = F_A \wedge \psi.$$
In this setting we can write $A= a(t)$ in temporal gauge as before and $\Phi=\phi(t) \in \Omega^0(I_t , \Omega^0(M, \mathfrak{g}_P))$ as a $1$-parameter family of Higgs fields over $M$. Then, the family $(a(t),\phi(t))$ of connections and Higgs fields on $M$ gives rise to a $\GG_2$-monopole if and only if they satisfy:
\begin{align*}
J_t \dot{a}  =  -d_a \phi - \ast_t \left( F_a \wedge \gamma_2 \right)\quad\text{and}\quad
\dot{\phi}  =  \Lambda_t F_a. \nonumber
\end{align*}
The analysis of these equations for the Bryant--Salamon $\GG_2$-manifolds  \cite{Bryant1989} is carried out in \cite{Oliveira2014}.
\end{remark}

\subsection*{Hamiltonian flow}

\noindent We now turn to a more formal aspect of the theory, which has not yet been used in applications, but which we have decided to point out here in case it may be of use in the future. On each slice $M_t=\lbrace t \rbrace \times M$, we may define a functional $\mathcal{F}_t$ on $\mathcal{A}$, the space of connections $a$ on $P$, by
$$\mathcal{F}_{t}(a) = \frac{1}{2} \int_{M_t} \langle F_a \wedge F_a \rangle \wedge \eta(t),$$
where $\eta(t) = \int^t \omega(s) ds$ and the $\langle \cdot , \cdot \rangle$ stands for an $\Ad$-invariant inner product on $\mathfrak{g}_P$, here applied to the $\mathfrak{g}_P$ components of the curvature in both entries. Then, given this $1$-parameter family of functionals $\mathcal{F}_t$, which we may also interpret as a single time-dependent functional, we may compute its gradient with respect to the time-dependent $L^2$-inner product induced by $g_t$.  We see that
\begin{align}\nonumber
\frac{d}{ds}\big\vert_{s=0} \mathcal{F}_{t}(a+sb) & =  \int_{M_t} \langle d_a b \wedge F_a \rangle \wedge \eta(t) \\ \nonumber
& =  \int_{M_t} d( \langle b \wedge F_a \rangle \wedge \eta(t) ) + \langle b \wedge F_a \rangle \wedge d\eta(t) \\ \label{eq:Gradient}
& =  \int_{M_t} \langle b \wedge F_a \rangle \wedge d\eta(t),
\end{align}
by Stokes' theorem. Moreover, using Hitchin's flow equations \eqref{eq:Hitchinflow}, 
$$d \eta(t) = \int^t d \omega(s) ds = \int^t \frac{\partial \gamma_1}{ds} ds = \gamma_1(t),$$
and so the outcome of the computation \eqref{eq:Gradient} is that the gradient of $\mathcal{F}_t$, with respect to the time-dependent $L^2$-inner product induced by $g_t$ on $\mathcal{A}$, is 
$$\nabla \mathcal{F}_t = \ast_t (F_a \wedge \gamma_1(t)).$$
At this point it is convenient to equip the space of connections on $P$ over each $M_t$ with a time-dependent (almost)-symplectic form given by $$\omega^{\mathcal{A}}_t(b_1,b_2)=\langle J_t b_1 , b_2 \rangle_{L^{2}(g_t)},$$ for $b_1,b_2$ two $\mathfrak{g}_P$-valued $1$-forms. Then, the Hamiltonian flow of $\mathcal{F}_t$ is $-J_t \nabla \mathcal{F}_t$ and we can regard the flow equation \eqref{eq:IEvolution} for $\GG_2$-instantons as the Hamiltonian flow of the time-dependent Hamiltonian $\mathcal{F}_t$. Thus, define the space of connections whose curvature in orthogonal to $\omega_t$ by 
$$\mathcal{A}_t = \lbrace a \in \mathcal{A} \ | \ \Lambda_t F_a =0 \rbrace .$$
We have shown in proposition \ref{prop:Compatible_Constraint} that the flow equation \ref{eq:IEvolution} starting at a connection in $\mathcal{A}_{0}$ always lies in $\mathcal{A}_t$. Putting this together with the discussion above, we have shown the following.

\begin{proposition} On $I_t \times M$, 
$\GG_2$-instantons are the solutions to the time-dependent Hamiltonian flow of $\mathcal{F}_t$ on $(\mathcal{A}, \omega^{\mathcal{A}}_t)$ starting at time $t=0$ in $\mathcal{A}_0$.
\end{proposition}

\section{Asymptotically conical (AC) \texorpdfstring{$\GG_2$}{G2}-manifolds}

In this section, we survey the known results on $\GG_2$-instantons on $\GG_2$-manifolds $X$ which are \emph{asymptotically conical} (AC); i.e.~$X$ is complete with one\footnote{A complete non-compact Ricci-flat manifold, which is not an isometric product, can only have one end due to the Cheeger--Gromoll splitting theorem.} non-compact end where the $\GG_2$-structure is asymptotic to a conical $\GG_2$-structure on $\mathbb{R}^+\times M$, 
 as given in \eqref{eq:G2cone}, for  some nearly K\"ahler structure $(\omega,\gamma_2)$ on $M$. \\
 \noindent It follows from Proposition 3 in \cite{Oliveira2014} (or easily from \eqref{eq:MEvolution1}-\eqref{eq:MEvolution2}) that on an AC $\GG_2$-manifold, a $\GG_2$-instanton whose curvature is decaying pointwise at infinity will have as a limit (if it exists) a pseudo-Hermitian--Yang--Mills connection $a$ (or nearly K\"ahler instanton) on $M$: i.e.~the curvature $F_a$ of $a$ satisfies 
 \begin{equation*}F_{a}\wedge\omega^2=0\quad\text{and}\quad F_{a}\wedge\gamma_2=0.
 \end{equation*} 
\noindent The known explicit examples of AC $\GG_2$-holonomy metrics (up to scale) are due to Bryant--Salamon \cite{Bryant1989}.  These metrics are either defined on the total space of the bundle of anti-self-dual 2-forms  on a self-dual Einstein 4-manifold with positive scalar curvature, or on $\mathbb{R}^4\times S^3$ (viewed as the spinor bundle of $S^3$).  These examples are cohomogeneity-1, and thus have a lot of symmetry, and so it is natural to look for $\GG_2$-instantons with symmetries on these AC $\GG_2$-manifolds.\\
In this section, we describe results from \cites{Oliveira2014,Lotay2018} which provide examples of $\GG_2$-instantons on the explicitly known AC $\GG_2$-manifolds.  We also review the results from \cite{Lotay2018} about the properties of the moduli space of $\GG_2$-instantons constructed on $\mathbb{R}^4\times S^3$.  This forms the content of $\S$\ref{ss:BS.N4}--\ref{ss:BS.R4xS3}. We conclude the section, in $\S$\ref{ss:AC.probs}, with some open problems we believe are worthy of investigation concerning $\GG_2$-instantons in this AC setting.

\subsection{On the Bryant--Salamon manifolds \texorpdfstring{$\Lambda^2_- (N^4)$}{Lambda2-(N4)}}\label{ss:BS.N4}

Let $(N^4,g_N)$ be a self-dual Einstein 4-manifold  with positive scalar curvature. Then $N$ is either $S^4$ or $\mathbb{CP}^2$ with $g_N$ being respectively either the round or Fubini--Study metric.  The AC Bryant--Salamon metric on the total space of the bundle of anti-self-dual 2-forms $X=\Lambda^2_-(N)$ on $N$ is such that 
 the zero section $N\subset \Lambda^2_-(N)$ is the unique compact coassociative submanifold in $X$ (in fact, any compact minimal submanifold in $X$ is contained in $N$ by Theorem 5.5 in \cite{TsaiWang}). If $\pi: \Lambda^2_-(N) \rightarrow N$ denotes the projection (this is the radially extended twistor projection, as the unit sphere bundle in $\Lambda^2_-(N)$ can be identified with the twistor space of $N$), then the Bryant--Salamon metric can be written as
$$g = f^2(s) g_{\mathbb{R}^3} +f^{-2}(s) \pi^*  g_N,$$
where $g_{\mathbb{R}^3}$ is the Euclidean metric along the fibers, $$f(s)=(1+s^2)^{-1/4}$$ and $s$ is the Euclidean distance along the fibers to the zero section. The geodesic distance to the zero section in the metric $g$ is $t(s)=\int_0^s f(u) du$ and using it we can write the metric as
$$g = dt^2 + s^2(t) f^2(s(t)) g_{S^2} +f^{-2}(s(t)) \pi^*  g_N,$$
where $g_{S^2}$ is the round metric in the unit normal spheres to $N$ (the twistor spheres). 

\subsubsection{$N=S^4$}

There is a cohomogeneity-$1$ action of $\mathrm{Sp}(2)$ on $\Lambda^2_-(S^4)$ whose principal orbits are the distance sphere bundles over $S^4$, which are diffeomorphic to the twistor space
$$\mathbb{CP}^3 = \mathrm{Sp}(2) / ( \mathrm{Sp}(1) \times \mathrm{U}(1) ).$$
We shall fix a reductive splitting
$$\mathfrak{sp}(2) = \mathfrak{h} \oplus \mathfrak{m},$$
as follows. Start by writing $\mathfrak{sp}(2) = \mathfrak{m}_1 \oplus \mathfrak{sp}_1(1) \oplus \mathfrak{sp}_2(1)$ and introduce a basis for the dual $\mathfrak{sp}(2)^*$ with
\begin{equation}
\mathfrak{m}_{1}^* = \langle e^1, e^2, e^3, e^4 \rangle \ , \ \mathfrak{sp}^*_1(1) =  \langle \eta^1, \eta^2, \eta^3 \rangle \  , \  \mathfrak{sp}^*_2(1) = \langle   \omega^1 , \omega^2, \omega^3 \rangle,
\end{equation}
where the $\eta^i$, $\omega^i$ form a standard dual basis for $\mathfrak{sp}(1) \cong \mathfrak{su}(2)$. Using the notation $e^{12}= e^1 \wedge e^2$, define the $2$-forms:
\begin{equation}
\begin{split}
\Omega_1  &=  e^{12} - e^{34}  \ , \quad  \Omega_2 = e^{13} - e^{42}  \ , \quad \Omega_3 =e^{14} - e^{23}\ ; \\ 
 \overline{\Omega}_1 & =  e^{12} + e^{34}   \ ,\quad  \overline{\Omega}_2 = e^{13} + e^{42}  \ , \quad \overline{\Omega}_3 =e^{14} + e^{23} \ .
\end{split}
\end{equation}
The Maurer--Cartan relations yield
\begin{equation}
\begin{split}
d \omega^i  = - 2 \omega^{jk} + \frac{1}{2}\Omega_i \ , \quad   d \eta^i = -2 \eta^{jk} - \frac{1}{2} \overline{\Omega}_i \label{MC},
\end{split}
\end{equation}
for $i=1,2,3$ and $(i,j,k)$ denoting a cyclic permutation of $(1,2,3)$. Furthermore, the Maurer--Cartan relations for the $de$'s can be used to compute
\begin{equation}
d \Omega_i = 2 \epsilon_{ijk} \left(  \Omega_j \wedge \omega^k - \Omega_k \wedge \omega^j \right),
\end{equation}
for $i \in \lbrace 1,2,3 \rbrace$. Then, we pick the reductive decomposition $\mathfrak{sp}(2) = \mathfrak{h} \oplus \mathfrak{m}$, such that
\begin{align}\label{irreds}
\mathfrak{m}^* & = \mathfrak{m}_1 \oplus \mathfrak{m}_2 = \mathfrak{m}_1 \oplus \mathbb{R} \langle \omega^2 , \omega^3 \rangle \\
\mathfrak{h}^* & =  \mathfrak{sp}_1(1) \oplus \mathbb{R} \langle \omega^1 \rangle.
\end{align}
Upon fixing the identifications $\mathfrak{m} \cong T_p\mathbb{CP}^3$ and $\mathfrak{m}_1 \cong T_{\pi(p)} S^4$. The $2$-forms $\Omega_i$ (resp. $\overline{\Omega}_i$) form a basis for the anti-self-dual (resp. self-dual) $2$-forms at $\pi(p)$. 

\medskip

In the complement of the zero section $\Lambda^2_-(S^4) \backslash S^4 \cong \mathbb{R}^+ \times \mathbb{CP}^3$, the $\GG_2$-holonomy metric can be written as
\begin{equation}\nonumber
\tilde{g}= dt \otimes dt + a^2(t) \left( \omega^2 \otimes \omega^2 + \omega^3 \otimes \omega^3 \right) + b^2(t) \left( \sum_{i=1}^4 e^i \otimes e^i \right),
\end{equation}
where $a(s) = 2s f(s^2)$ and $b(s)= \sqrt{2} f^{-1}(s^2)$. A $\GG_2$-structure giving rise to this metric can be written as
\begin{equation}\nonumber
\varphi  =  dt \wedge \left( a^2 \omega^{23} + b^2 \Omega_1 \right) + ab^2 \left( \omega^3 \wedge \Omega_2 - \omega^2 \wedge \Omega_3 \right) ,
\end{equation}
and
\begin{equation}\label{eq:psi_AC_BS_Lambda}
\psi  =  b^4 e^{1234} -a^2 b^2 \omega^{23} \wedge \Omega_1 - ab^2 dt \wedge \left( \omega^2 \wedge \Omega_2 + \omega^3 \wedge \Omega_3 \right).
\end{equation}
We now consider the bundle
$$P_{\lambda}= \mathrm{Sp}(2) \times_{(\lambda, \mathrm{Sp}(1) \times \mathrm{U}(1) )} \mathrm{SU}(2),$$
where $\lambda: \mathrm{Sp}(1) \times \mathrm{U}(1)  \rightarrow \mathrm{SU}(2)$ is given by $\lambda(g , e^{i \theta})= \diag( e^{il\theta}, e^{-il \theta})$, for some $l \in \mathbb{Z}$ and $(g , e^{i \theta}) \in \SU_1(2) \times \mathrm{U}_2(1)$. There is a canonical invariant connection, which as a $1$-form in $\mathrm{Sp}(2)$ with values in $\mathfrak{su}(2)$ can be written as 
$$A_c = \omega^1 \otimes T_1,$$ 
where $T_1,T_2,T_3$ is a standard basis for $\mathfrak{su}(2)$. Then, one can prove that (up to an invariant gauge transformation) any other connection $A \in \Omega^1(\mathrm{Sp}(2), \mathfrak{su}(2))$ can be written as $A= A_c + (A- A_c)$ with
\begin{equation}\label{Lambda}
A - A_c  =  a \left( T_2 \otimes \omega^2  +   T_3 \otimes \omega^3 \right) ,
\end{equation}
with $a \in \mathbb{R}$.\\
Now we consider the bundle $P$ pulled back to $\Lambda^2_-S^4 \backslash S^4 \cong \mathbb{R}^+ \times \mathbb{CP}^3$ and equip it with an invariant connection $A \in \Omega^1( \mathbb{R}^+_{t} \times \mathrm{Sp}(2), \mathfrak{su}(2))$ in radial gauge, i.e.~$A(\partial_s)=0$. Thus $A$ must be a $1$-parameter family of connections as above. This is determined by $a$ which is now a real-valued function of $t \in \mathbb{R}^+$, as it must be constant along any $\mathrm{Sp}(2)$ orbit. 
A straightforward computation yields that the curvature $F_A$ of the connection 
 $A$
satisfies the $\GG_2$-instanton equation $F_A \wedge \psi=0$ if and only if 
\begin{eqnarray}\nonumber
s^2 f^{4} =  1- a^2 , \quad \frac{da}{ds} = - sf^{-4} a . 
\end{eqnarray}
In terms of $ t(s)= \int_0^s  f(l^2)dl = \int_0^s  \left( 1 + l^2 \right)^{-\frac{1}{4}} dl$, the second of these is
\begin{equation}\label{monO2}
\frac{da}{dt} = - s f^{-3} a.
\end{equation}
Moreover, solving the first equation, which is algebraic,  yields
$$a(t)= \pm f^2(s(t)),$$ 
which one can check does provide a solution of the ODE \eqref{monO2}. This proves the following result.

\begin{theorem}\label{insttheorem}
The $\mathrm{SU}(2)$ connection $$A = A_c \pm (1+s^2)^{-\frac{1}{2}} \left( T_2 \otimes \omega^2 + T_3 \otimes \omega^3 \right)$$ on $P \rightarrow \Lambda^2_-(S^4)$ is an irreducible $\GG_2$-instanton, with curvature given by
\begin{align*}\nonumber
F_A  = &  \left( \frac{\Omega_1}{2}  - \frac{2s^2}{1 + s^2}\omega^{23} \right) \otimes T_1  \pm \frac{1}{2\sqrt{1+s^2}}\left( \Omega_2 \otimes T_2 + \Omega_3 \otimes T_3 \right) \\ \nonumber                      
 &  \mp \frac{s}{1+s^2} \left( ds \wedge \omega^2 \otimes T_2+ ds \wedge \omega^3 \otimes T_3\right).
\end{align*}
\end{theorem}

\begin{remark}
These instantons are asymptotic to the canonical invariant connection $A_c$. This is a $t$-independent reducible connection which is in fact pseudo-Hermitian--Yang--Mills with respect to the standard nearly K\"ahler structure on $\mathbb{CP}^3$.
\end{remark}

The Levi-Civita connection of the round metric induces a self-dual connection in the $\mathrm{Spin}$ bundle over $S^4$. Lifting this to $\Lambda^2_-(S^4)$ also gives rise to a $\GG_2$-instanton. To prove this we must construct the Spin bundle 
$$Q = \mathrm{Sp}(2) \times_{(\mu , \mathrm{Sp}(1) \times \mathrm{U}(1))} \mathrm{Sp}(1),$$
where $\mu: \mathrm{Sp}(1) \times \mathrm{U}(1) \rightarrow \mathrm{Sp}(1) \cong \mathrm{SU}(2)$ is simply the projection on the first component.
The canonical invariant connection in $Q$ is the Spin connection and is given by 
$$\theta= \eta^1 \otimes T_1 + \eta^2 \otimes T_2 + \eta^3 \otimes T_3.$$ 
Using the Maurer--Cartan relations \eqref{MC} one can compute the curvature  to be 
\begin{align*}\nonumber
F_{\theta} & =  d\theta + \frac{1}{2} [ \theta \wedge \theta ] \\ \nonumber
& =  2 \eta^{23} \otimes T_1 + 2 \eta^{31} \otimes T_2 + 2 \eta^{12} \otimes T_3 \\ \nonumber
& \quad - \left( 2 \eta^{23} + \frac{1}{2} \overline{\Omega}_1 \right) \otimes T_1 - \left( 2 \eta^{31} + \frac{1}{2} \overline{\Omega}_2 \right) \otimes T_2 - \left( 2 \eta^{12} + \frac{1}{2} \overline{\Omega}_3 \right) \otimes T_3. 
\end{align*}
We shall state this as follows.

\begin{proposition}\label{prop:Spin_Connection}
The lift of the Spin connection $\theta$ on $S^4$ to $\Lambda^2_-(S^4)$ is a $\GG_2$-instanton with curvature
$$F_{\theta}= - \frac{1}{2} \overline{\Omega}_1  \otimes T_1  - \frac{1}{2} \overline{\Omega}_2 \otimes T_2 - \frac{1}{2} \overline{\Omega}_3 \otimes T_3.$$
\end{proposition}

\begin{remark}
Proposition \ref{prop:Spin_Connection} is a consequence of a more general phenomena. Indeed, for $N$ either $\mathbb{CP}^2$ or $\mathbb{S}^4$, the pullback of any self-dual connection on $N$ gives rise to a $G_2$-instanton on the Bryant-Salamon $G_2$-manifolds $\Lambda^2_-(N)$. This can be seen immediately from the calibrating $4$-form $\psi$ in equation \ref{eq:psi_AC_BS_Lambda} as noticed in \cite{Oliveira2014}.
\end{remark}

\subsubsection{$N= \mathbb{CP}^2$}

As already remarked above, the sphere bundle in $\Lambda^2_-( \mathbb{CP}^2)$ is the twistor space of $\mathbb{CP}^2$, which is the flag manifold $\mathbb{F}_2$. This is homogeneous and $\mathrm{SU}(3)$ acts transitively with isotropy the maximal torus $\mathrm{U}(1)^2$. The Serre spectral sequence for the fibration $\mathrm{SU}(3) \rightarrow \mathbb{F}_2$ gives $H^2( \mathbb{F}_2 , \mathbb{Z}) \cong H^1(\mathrm{U}(1)^2, \mathbb{Z})$, which we can further identify with the integral weight lattice in $(\mathfrak{u}(1)^2)^*$. An explicit way to make the identification is as follows. Given an integral weight $\alpha \in (\mathfrak{u}(1)^2)^*$ we construct the line bundle on $\mathbb{F}_2$
$$L_{\alpha} = \mathrm{SU}(3) \times_{(e^{\alpha}, \mathrm{U}(1)^2)} \mathbb{C}.$$
Now let $1 \in \mathrm{SU}(3)$ be the identity and $\mathfrak{m}\subset \mathfrak{su}(3)$ be a reductive complement to the Cartan subalgebra generated by the isotropy, i.e. $\mathfrak{su}(3) = \mathfrak{u}(1)^2 \oplus \mathfrak{m}$ with $[\mathfrak{u}(1)^2 , \mathfrak{m}] \subset \mathfrak{m}$ (for example, we can let $\mathfrak{m}$ be the real part of the root spaces). Then, we extend $\alpha$, first to $\mathfrak{su}(3)^*$ by letting it vanish on $\mathfrak{m}$, and secondly to $\Omega^1(\mathrm{SU}(3), i\mathbb{R})$ by left translations. It is now easy to see that $\alpha$ equips $L_{\alpha}$ with a connection and so its first Chern class $ \frac{i}{2 \pi}[d \alpha] \in H^2(\mathbb{F}_2 , \mathbb{Z})$ gives the corresponding element in the second cohomology induced by $\alpha$. The connection $\alpha$ is usually called the canonical invariant connection on $L_{\alpha}$ and is uniquely determined by $\mathfrak{m}$.\\
We shall now turn to the construction of $\mathrm{SO}(3)$-bundles over $\mathbb{F}_2$, carrying interesting invariant connections. These are constructed by composing the homomorphism $e^{ \alpha }: \mathrm{U}(1)^2 \rightarrow \mathrm{U}(1)$ with the embedding of $ \mathrm{U}(1) \hookrightarrow \mathrm{SO}(3)$ as the maximal torus, then setting
$$P_{\alpha} = \mathrm{SU}(3) \times_{(e^{ \alpha}, \mathrm{U}(1)^2)} \mathrm{SO(3)}.$$
These $\mathrm{SO}(3)$-bundles are in fact reducible to the circle bundles inducing $L_{\alpha}$ and can be equipped with the induced connections $ \alpha \in  \Omega^1(\mathrm{SU}(3), \mathfrak{so}(3))$ viewed as left invariant $1$-forms in $\mathrm{SU}(3)$ with values in $\mathfrak{so}(3)$ by embedding $i \mathbb{R} \hookrightarrow \mathfrak{so}(3)$. These induced connections are also $\mathrm{SU}(3)$-invariant and it follows from Wang's theorem, \cite{Wang1958}, that other invariant connections are in $1$-to-$1$ correspondence with morphisms of $\mathrm{U}(1)^2$-representations
$$\Lambda : ( \mathfrak{m}, \Ad ) \rightarrow (\mathfrak{so}(3), \Ad \circ e^ {\alpha} ).$$
Decompose these into irreducible components $ \mathfrak{m}  \cong \mathbb{C}_{\alpha_1} \oplus \mathbb{C}_{\alpha_2} \oplus \mathbb{C}_{\alpha_3}$, where $\alpha_1, \alpha_2, \alpha_3$ are the positive roots of $\mathrm{SU}(3)$, while $\mathfrak{so}(3) \cong \mathbb{R}_0 \oplus \mathbb{C}_{\alpha}$. Hence it follows from Schur's lemma that such morphisms of representations exist if and only if $\alpha$ is one of the roots, in which case $\Lambda$ restricts to the corresponding root space as an isomorphism onto $\mathbb{C}_{\alpha} \subset \mathfrak{so}(3)$ and vanishes in all other components. If $\alpha = \alpha_i$ we shall denote these by $\Lambda_i$. Then, notice that fixing a basis of $\mathfrak{m}$ and a basis of $\mathfrak{so}(3)$ (i.e. a gauge) each $\Lambda_i$ is determined up to a constant.\\
The problem of constructing instantons on the bundles $P_{\alpha}$ was analysed in \cite{Oliveira2014}. The first point to settle is that the bundle $P_{\alpha}$ on which one is solving the instanton equations must extend to a bundle over all of $\Lambda^2_-(\mathbb{CP}^2)$, i.e. across the zero section. It turns out that there is only one such $\alpha$, say $\alpha = \alpha_2$, which can be characterized by being in the image of the map
$$(\pi_2)^*: H^2(\mathbb{CP}^2, \mathbb{Z}^2) \rightarrow H^2(\mathbb{F}_2 , \mathbb{Z}^2),$$ 
where $\pi_2 : \mathbb{F}_2 \rightarrow \mathbb{CP}^2$ is the twistor projection. Thus, take $\alpha= \alpha_2$, and extend the bundle and the connection to the whole of $\Lambda^2_-(\mathbb{CP}^2)$. Now the connection 
$$A=\alpha + \Lambda_2 (r)$$ 
can be seen as an element of $\Omega^1(\mathbb{R}^+ \times \mathrm{SU}(3), \mathfrak{so}(3))$. Then, in \cite{Oliveira2014} the invariant instanton equations for $A$ are computed, very much in the same way as the case of $\Lambda^2_-(S^4)$ above. They appear as an ODE and an algebraic equation for $\vert \Lambda_2 \vert$, with the ODE being implied by the algebraic equation which is
$$2 s^2( r ) f^{-2}(r) \vert \Lambda_{2} \vert^2=1. $$ 
In order to explicitly write this connection we fix a standard basis $\lbrace T_1 , T_2, T_3 \rbrace$ of $\mathfrak{so}(3)$ so that the image of $\alpha$ is parallel to $T_1$. Then, the complement $\mathbb{C}_{\alpha} \subset \mathfrak{so}(3)$ is generated by $T_2,T_3$, and there are left-invariant $1$-forms $\nu_1 , \nu_2$ on $\mathrm{SU}(3)$ such that the restriction to the tangent space to the identity of the map
$$\nu_1 \otimes T_2 + \nu_2 \otimes T_3 \vert_{\mathbb{C}_{\alpha_2}} : \mathbb{C}_{\alpha_2} \subset \mathfrak{m} \subset \mathfrak{su}(3) \rightarrow \mathbb{C}_{\alpha} \subset \mathfrak{so}(3),$$
is an isomorphism. Furthermore, as in the case of $\Lambda^2_-(S^4)$ we fix $\Omega_1, \Omega_2, \Omega_3$ a universal basis for the anti-self-dual $2$-forms on $\mathbb{CP}^2$. These are chosen so that $\sum_{i=1}^3 \frac{\text{scal}}{24} \Omega_i \otimes T_i$ is the curvature of the Levi-Civita induced connection on $\Lambda^2_-$. Then, we can write the $\GG_2$-instanton $A$ as in the following result.

\begin{theorem}\label{insttheorem2}
The connection on $P_{\alpha_2}$ over $\Lambda^2_-(\mathbb{CP}^2)$ given by $$A = \alpha \pm (1+s^2)^{-\frac{1}{2}}\left( \nu_1 \otimes T_2 + \nu_2 \otimes T_3 \right)$$ is an irreducible $\GG_2$-instanton with curvature
\begin{align*}\nonumber
F_{A}  = & \frac{2s^2}{s^2 + 1} \nu_{12}  \otimes T_1 +\Omega_1 \otimes T_1 \pm \frac{1}{\sqrt{s^2 + 1}}\left( \Omega_2 \otimes T_2 + \Omega_3 \otimes T_3 \right) \\  \nonumber
 &  \mp \frac{s}{(1 + s^2 )^{\frac{3}{2} } } \left(  ds \wedge \nu_1 \otimes T_2 + ds \wedge \nu_2 \otimes T_3 \right).
\end{align*}
\end{theorem}

\begin{remark}
This instanton converges (at a polynomial rate) to the canonical invariant connection $\alpha$, which is the pullback to the cone on $\mathbb{F}_2$ of a reducible pseudo-Hermitian--Yang--Mills connection on $\mathbb{F}_2$ equipped with its standard nearly K\"ahler structure.
\end{remark}

In \cite{Oliveira2014} irreducible $\GG_2$-instantons with gauge group $G=\mathrm{SU}(3)$ in this setting are also investigated. For this we consider the bundle 
$$Q= \mathrm{SU}(3) \times_{\mathrm{U}(1)^2} \mathrm{SU}(3),$$
where $\mathrm{U}(1)^2$ acts diagonally on both $\mathrm{SU}(3)$ factors by fixing a maximal torus. As before we decompose $\mathfrak{su}(3)$ into irreducible $\mathfrak{u}(1)^2$ representations, as
$$\mathfrak{su}(3) = \mathfrak{u}(1)^2 \oplus \mathbb{C}_{\alpha_1} \oplus \mathbb{C}_{\alpha_2} \oplus \mathbb{C}_{\alpha_3}.$$
Then, we fix certain isomorphisms $l: \mathfrak{u}(1)^2 \rightarrow \mathfrak{u}(1)^2$ and $\lambda_i: \mathbb{C}_i \rightarrow \mathbb{C}_i$  , which we interpret as being left-invariant maps from (subspaces of) $T_1 \mathrm{SU}(3) \cong \mathfrak{su}(3) \rightarrow \mathfrak{su}(3) \cong \mathfrak{g}$, i.e.~as left-invariant $1$-forms on $\mathrm{SU}(3)$ with values in the Lie algebra of the gauge group $\GG=\mathrm{SU}(3)$. Then,  Theorem 9 in \cite{Oliveira2014} can be written in the following way.

\begin{theorem}\label{SU(3)instantons}
There are two real $1$-parameter families of irreducible $\GG_2$-instantons on $Q$ parametrized by $c \geq 0$. These are given by
$$
A =  l - \frac{ u_{c}(s)}{\sqrt{1+s^2}}\lambda_2 \mp \frac{\sqrt{u_c^2(s)-1}}{s} \left(\lambda_3 - \lambda_1 \right)
$$
and
$$
A  =  l + \frac{ u_{c}(s)}{\sqrt{1+s^2}} \lambda_2 \mp \frac{\sqrt{u_c^2(s)-1}}{s} \left( \lambda_3 + \lambda_1 \right),
$$
where
\begin{equation}\nonumber
u_c(s)= 1 - 2c \frac{ s^2 }{s^2(1+c) + 2 \left( \sqrt{1+ s^2} +1 \right)} .
\end{equation}
In particular, the case $c=-1$ gives flat connections. 
\end{theorem}

\subsection{On the Bryant--Salamon \texorpdfstring{$\mathbb{R}^4\times S^3$}{R4xS3}}\label{ss:BS.R4xS3}

The Bryant--Salamon metric on $\mathbb{R}^4\times S^3$ \cite{Bryant1989} is 
$\SU(2)^2\times \U(1)$-invariant  and so we are motivated to study $\GG_2$-instantons with the same symmetry: in fact, the metric is $\SU(2)^3$-invariant, but it convenient to take the $\SU(2)^2\times \U(1)$-invariant point of view for later study.

\subsubsection{\texorpdfstring{$\SU(2)^2\times \U(1)$-symmetry}{SU(2)xSU(2)xU(1)-symmetry}}\label{ss:symmetry}
We begin with some preparation for studying $\SU(2)^2$-invariant holonomy $\GG_2$-metrics and instantons. Split the Lie algebra $\mathfrak{su}(2) \oplus \mathfrak{su}(2)$ as $\mathfrak{su}^+ \oplus \mathfrak{su}^-$, as follows.
 If $\lbrace T_i \rbrace_{i=1}^3$ is a basis for $\mathfrak{su}(2)$ such that $[T_i,T_j] = 2 \epsilon_{ijk} T_k$, 
then $T_i^+ = (T_i, T_i)$ and $T_i^-=(T_i, -T_i)$ for $i=1,2,3$ give a basis for $\mathfrak{su}^+$ and $\mathfrak{su}^-$ respectively.  (Thus $\mathfrak{su}^+$ and $\mathfrak{su}^-$ are 
diagonal and anti-diagonal copies of $\mathfrak{su}(2)$ in $\mathfrak{su}(2)\oplus\mathfrak{su}(2)$.) We shall let $\lbrace \eta_i^+ \rbrace_{i=1}^{3}$ and $\lbrace \eta_i^- \rbrace_{i=1}^{3}$ be dual bases to $\{T_i^+\}_{i=1}^3$ and $\{T_i^-\}_{i=1}^3$ respectively. The Maurer--Cartan relations in this case give
\begin{eqnarray}
d\eta_i^+ & = & -\epsilon_{ijk} \left( \eta_j^+ \wedge \eta_k^+ + \eta_j^- \wedge \eta_k^- \right), \label{eq:MC1}\\ 
d\eta_i^- & = & -2\epsilon_{ijk} \eta_j^- \wedge \eta_k^+.\label{eq:MC2}
\end{eqnarray}
The complement of the singular orbit can be written as $\mathbb{R}^+_t \times M$, where $M$ 
denotes a principal orbit, which is a finite quotient of $S^3\times S^3$ (for the Bryant--Salamon metric, it will simply be $S^3\times S^3$). 
The $\SU(2)\times \SU(2)$-invariant $\SU(3)$-structure on the principal orbit $\{t\}\times M$ is given by (\cite{Madsen2013})
\begin{eqnarray}
\omega & = & 4 \sum_{i=1}^3 A_iB_i \eta_i^- \wedge \eta_i^+, \label{eq:BBSU(3)structure1}\\ 
\gamma_1 & = & 8 B_1B_2B_3 \eta_{123}^- - 4 \sum_{i,j,k}  \epsilon_{ijk} A_i A_j B_k \eta_i^+ \wedge \eta_j^+ \wedge \eta_k^-, \\ \label{eq:BBSU(3)structure}
\gamma_2 & = & - 8 A_1A_2A_3 \eta_{123}^+ + 4 \sum_{i,j,k}  \epsilon_{ijk} B_i B_j A_k \eta_i^- \wedge \eta_j^- \wedge \eta_k^+,\label{eq:BBSU(3)structure2}
\end{eqnarray}
for real-valued functions $A_i$, $B_i$ of $t \in \mathbb{R}^+$, where $\eta_{123}^{\pm}$
 denotes $\eta_{1}^{\pm} \wedge \eta_{2}^{\pm} \wedge \eta_{3}^{\pm}$. 
 The compatible metric determined by this $\SU(3)$ structure on $\{t\}\times M$ is  (\cite{Madsen2013})
\begin{equation}\label{eq:metric}
g_t = \sum_{i=1}^3 (2A_i)^2 \eta_{i}^+ \otimes \eta_i^+ + (2B_i)^2 \eta_i^- \otimes \eta_i^-,
\end{equation}
and the resulting metric on $\mathbb{R}_t \times M$, compatible with the $\GG_2$-structure $\varphi = dt \wedge \omega + \gamma_1$, is given by $g=dt^2 + g_t$. Recall also that this metric has holonomy in $\GG_2$ if and only if the $\SU(3)$-structure above solves the Hitchin flow equations \eqref{eq:Hitchinflow}.\\
All known complete $\SU(2)^2$-invariant holonomy $\GG_2$ metrics have an extra $\U(1)$-symmetry: this $\U(1)$ acts diagonally on $S^3 \times S^3$ with infinitesimal generator $T_1^+$. 
   As a consequence, we have $A_2=A_3$ and $B_2=B_3$ and \eqref{eq:Hitchinflow} becomes (as in \cite{Bazaikin2013}):
\begin{align}
\dot{A}_1\;\, &= \;\, \frac{1}{2}\left(\frac{A_1^2}{A_2^2}-\frac{A_1^2}{B_2^2}\right), \label{eq:dotA1}\\
 \dot{A}_2\;\, &= \;\, \frac{1}{2}\left(\frac{B_1^2+B_2^2-A_2^2}{B_1B_2}-\frac{A_1}{A_2}\right),\label{eq:dotA2} \displaybreak[0]\\
\dot{B}_1\;\, &= \;\, \frac{A_2^2+B_2^2-B_1^2}{A_2B_2}, \label{eq:dotB1}\displaybreak[0]\\
\dot{B}_2\;\, &= \;\,\frac{1}{2}\left(\frac{A_2^2+B_1^2-B_2^2}{A_2B_1}+\frac{A_1}{B_2}\right).\label{eq:dotB2}
\end{align}  

\subsubsection{The Bryant--Salamon metric}
 As we stated above, the Bryant--Salamon metric on $\mathbb{R}^4\times S^3$  is actually $\SU(2)^3$-invariant: the principal orbits are $\SU(2)^3/\SU(2)\cong S^3\times S^3$ and the (unique) singular orbit is  $\SU(2)^3/\SU(2)^2\cong S^3$.  (Here, the $\SU(2)$ in $\SU(2)^3$ is the subgroup $\SU(2)_3 = 1 \times 1 \times \SU(2)$, and $\SU(2)^2 \subset \SU(2)^3$ is the subgroup $\Delta \SU(2) \times \SU(2) $, where $\Delta \SU(2) \subset \SU(2)^2 $ is the diagonal.) 
\\
In this case the extra symmetry means that $A_1=A_2=A_3$ and $B_1=B_2=B_3$ and the equations \eqref{eq:dotA1}-\eqref{eq:dotB2} reduce to:
\begin{equation}\label{eq:dotAdotB}
\dot{A}_1=\frac{1}{2}\left(1-\frac{A_1^2}{B_1^2}\right)\quad\text{and}\quad
\dot{B}_1=\frac{A_1}{B_1}.
\end{equation}
Setting $B_1=s$ and $A_1=sC(s)$ we see that \eqref{eq:dotAdotB} becomes
$\frac{d}{ds}(sC)=\frac{1-C^2}{2C}$ 
which we can easily solve as 
$C(s)=\sqrt{\frac{1-c^3s^{-3}}{3}}$, 
so that, for $c>0$ and $s\geq c$,
\begin{equation}\label{eq:AB.c.BS}
A_1(s)=\frac{s}{\sqrt{3}}\sqrt{1-c^3s^{-3}}\quad\text{and}\quad 
B_1(s)=s.
\end{equation}
In particular, choosing $c=1$ and using $t$, the arc length parameter along the geodesic parametrized by $s$, we define a coordinate $r \in [1, \infty)$ implicitly by 
\begin{equation}\label{eq:r.BS}
t(r)=\int_1^r \frac{ds}{\sqrt{1-s^{-3}}},
\end{equation} and solve \eqref{eq:dotAdotB} as follows:
\begin{equation}\label{eq:AB.BS}
A_1=A_2=A_3=\frac{r}{3}\sqrt{1-r^{-3}} \quad \text{and} \quad B_1=B_2=B_3=\frac{r}{\sqrt{3}}.
\end{equation}
It is easy to verify that the geometry at infinity is asymptotically conical to the standard holonomy $\GG_2$-cone on $S^3\times S^3$.  In fact, we see from \eqref{eq:AB.c.BS} that one obtains a one-parameter family\footnote{There are, in fact, distinct $\SU(2)^3$-invariant torsion-free $\GG_2$-structures on $\mathbb{R}^4\times S^3$ inducing the same asymptotially conical Bryant--Salamon metric, determined by their image in $H^3(S^3\times S^3)$.} of solutions to \eqref{eq:dotAdotB}, equivalent up to scaling, whose limit with $c=0$ is the conical solution.  Moreover, the torsion-free $\GG_2$-structure has a unique compact associative submanifold which is the singular orbit $S^3 $.


\subsubsection{Examples of \texorpdfstring{$\GG_2$-instantons}{G2-instantons}}

It is straightforward to write down the evolution equation \eqref{eq:IEvolution} for $\SU(2)^2$-invariant $\GG_2$-instantons on a $\U(1)$-bundle over the Bryant--Salamon $\mathbb{R}^4\times S^3$.  One can solve this equation explicitly and obtain the following result.

\begin{proposition}\label{prop:BS.U1}
Any $\SU(2)^2$-invariant $\GG_2$-instanton $A$ with gauge group $\U(1)$ over the Bryant--Salamon $\mathbb{R}^4\times S^3$  can be written as
\begin{equation*}
A 
=\frac{r^3-1}{r}\sum_{i=1}^3x_i\eta_i^+
\end{equation*}
for some $x_1,x_2,x_3\in\mathbb{R}$, where $r\in[1,+\infty)$ is determined 
by \eqref{eq:r.BS}. 
\end{proposition}

\noindent We therefore wish to turn to a non-abelian gauge group, namely $\SU(2)$.  The only possible homogeneous $\SU(2)$-bundle $P$ on the principal orbits $S^3\times S^3$ is $P=\SU(2)^2 \times \SU(2)$, i.e.~the trivial $\SU(2)$-bundle.  We therefore consider connections on this bundle with the $\SU(2)^3$-symmetry existent in the underlying Bryant--Salamon geometry, and derive the following evolution equations for invariant $\GG_2$-instantons in this setting from \eqref{eq:IEvolution} (after some work).

\begin{proposition}\label{prop:BSODEs}
Let $A$ be an $\SU(2)^3$-invariant $\GG_2$-instanton  with gauge group $\SU(2)$ on $\mathbb{R}^+ \times \SU(2)^2 \cong \mathbb{R}^+ \times \SU(2)^3/ \Delta \SU(2)$. There is a standard basis $\{T_i\}$ of $\mathfrak{su}(2)$, 
i.e.~with $[T_i , T_j]= 2 \epsilon_{ijk} T_k$, such that (up to an invariant gauge transformation) we can write
\begin{equation}\label{eq:ClarkeA}
A=A_1x\left( \sum_{i=1}^3 T_i\otimes \eta_i^+ 
\right) + B_1y\left( \sum_{i=1}^3 T_i\otimes\eta_i^- 
\right),
\end{equation}
with $x$, $y:\mathbb{R}^+\to\mathbb{R}$ satisfying 
\begin{align}\label{eq:Clarke1}
\dot{x} & \; = \; \frac{\dot{A}_1}{A_1} x + y^2-x^2  && \hspace{-24pt}   = \; \frac{1}{2A_1}\left(1-\frac{A_1^2}{B_1^2}\right)x+y^2-x^2, && \\ \label{eq:Clarke2}
\dot{y} & \; = \;  \frac{ 2\dot{A}_1 -3}{A_1} y  +2xy  && \hspace{-24pt} =   \;
-\frac{1}{A_1}\left(2+\frac{A_1^2}{B_1^2}\right)y+2xy.&&
\end{align}
\end{proposition}

\noindent Next we must determine the initial conditions in order for an $SU(2)^3$-invariant $G_2$-instanton $A$, given by a solution to the ODEs in Proposition \ref{prop:BSODEs}, to extend smoothly over the singular orbit $S^3=SU(2)^2/ \Delta SU(2)$.  For that we need to first extend the bundle over the singular orbit. Up to an isomorphism of homogeneous bundles, there are two possibilities: these are
\begin{equation}\label{eq:Plambda}
P_{\lambda}=SU(2)^2 \times_{(\Delta SU(2), \lambda)} SU(2),
\end{equation}
with the homomorphism $\lambda:SU(2)\rightarrow SU(2)$ being either the trivial one (which we denote by $1$) or the identity $\id$. Depending on the choice of $\lambda$, the conditions for the connection $A$ to extend are different, as we show in the following lemma.
\begin{lemma}\label{lem:SmoothlyExtendBS}
The connection $A$ in \eqref{eq:ClarkeA} extends smoothly over the singular orbit $S^3$ if $x(t)$ is odd, $y(t)$ is even, and their Taylor expansions around $t=0$ are 
\begin{itemize}
\item either $x(t)=x_1t+ x_3 t^3+\ldots , \ y(t)= 
y_2 t^2 + \ldots$, in which case $A$ extends smoothly 
as a connection on $P_{1}$; 

\item or $x(t)=\frac{2}{t}+x_1t+\ldots , \ y(t)=y_0 + y_2 t^2 +\ldots$, in which case $A$ extends smoothly as a connection on $P_{\id}$.
\end{itemize}
\end{lemma}

\noindent 
If we set $y=0$ in the notation of Proposition \ref{prop:BSODEs}, the ODEs there become the single ODE:
\begin{equation}\label{eq:dotxBS}
\dot{x}=\frac{\dot{A}_1}{A_1} x-x^2.
\end{equation}
Writing this equation as 
\begin{equation}\label{eq:dotxBS2}
\frac{d}{dt}\left( \frac{x}{A_1} \right) = -A_1 \left( \frac{x}{A_1} \right)^2
\end{equation} makes it separable.  Since $B_1\dot{B}_1=A_1$ by \eqref{eq:dotAdotB} and $B_1^2(0)=\frac{1}{3}$, \eqref{eq:dotxBS2} can be readily integrated to show that 
\begin{equation}
x(t)=\frac{2x_1 A_1(t)}{1+ x_1 (B_1^2(t)-\frac{1}{3})}.
\end{equation}
We can explicitly see from Lemma \ref{lem:SmoothlyExtendBS} that the connection $A$ extends smoothly over $S^3$ as a connection on $P_1$.  This is precisely the one-parameter family of $SU(2)^3$-invariant $\GG_2$-instantons on the Bryant--Salamon $\mathbb{R}^4\times S^3$ constructed by Clarke \cite{Clarke14}, and the parameter can be interpreted as how concentrated the instanton is around the associative $S^3$.\\
In fact, it is shown in \cite{Lotay2018} that these are the only irreducible $\SU(2)^2\times \U(1)$-invariant $\GG_2$-instantons on $P_1$: in particular, this shows that all irreducible $\SU(2)^2\times\U(1)$-invariant $\GG_2$-instantons on $P_1$ on the Bryant--Salamon $\mathbb{R}^4\times S^3$ are actually $\SU(2)^3$-invariant.

\begin{theorem}\label{prop:Clarke}\label{thm:Clarke}
The moduli space $\mathcal{M}_{P_1}^{BS}$ of irreducible $\SU(2)^2\times\U(1)$-invariant $\GG_2$-instantons with gauge group $\SU(2)$ defined on $P_1$ on the Bryant--Salamon $\mathbb{R}^4\times S^3$  is parametrized by the open interval $(0,+\infty)$.\\
Specifically, let $A$ be an $\SU(2)^2\times\U(1)$-invariant $\GG_2$-instanton with gauge group $\SU(2)$ on the Bryant--Salamon $\mathbb{R}^4\times S^3$, which extends smoothly over the singular orbit on $P_{1}$. 
\begin{itemize}
\item[(a)]  
If $A$ is irreducible, then it is one of Clarke's examples \cite{Clarke14}, in which case it is $\SU(2)^3$-invariant and there is $x_1 \in \mathbb{R}$ such that, in 
the notation of Proposition \ref{prop:BSODEs},
%
%
\begin{equation*}
x(r)=\frac{2x_1r\sqrt{1-r^{-3}}}{3+x_1(r^2-1)}\quad\text{and}\quad y(r)=0,
\end{equation*}
where $r\in[1,+\infty)$ is determined by \eqref{eq:r.BS}.  That is, $A$ can be written as
%
%
\begin{equation*}
A^{x_1}=\frac{2x_1(r^3-1)}{3r\big(3+x_1(r^2-1)\big)}\left(\sum_{i=1}^3T_i\otimes\eta_i^+\right).
\end{equation*}  
Observe that $A^{x_1}$ is defined globally on $\mathbb{R}^4\times S^3$ if and only if $x_1\geq 0$ and that $A^0$ is the trivial flat connection.
\item[(b)] If $A$ is reducible, it has gauge group $\U(1)$ and is given in Proposition \ref{prop:BS.U1} with $x_2=x_3=0$, i.e.~
\begin{equation*}
A 
=\frac{r^3-1}{r}x_1\eta_1^+
\end{equation*}
for some $x_1\in\mathbb{R}$, where $r\in[1,+\infty)$ is as in \eqref{eq:r.BS}.
\end{itemize}
\end{theorem}

\noindent We now turn to $\SU(2)^2\times\U(1)$-invariant $\GG_2$-instantons defined on $P_{\id}$, for which we have a local existence result for a 1-parameter family of such $\GG_2$-instantons.

\begin{proposition}\label{prop:LocalExistenceBS_1}
Let $S^3$ be the singular orbit in the Bryant--Salamon $\mathbb{R}^4\times S^3$. There is a one-parameter family of $\SU(2)^2\times\U(1)$-invariant $\GG_2$-instantons, with gauge group $\SU(2)$, defined in a neighbourhood of $S^3$ and smoothly extending over $S^3$ on $P_{\id}$. The instantons are actually $\SU(2)^3$-invariant and parametrized by $y_0 \in \mathbb{R}$ satisfying, in the notation of Proposition \ref{prop:BSODEs},
$$x(t) = \frac{2}{t} + \frac{y_0^2-1}{4} t + O(t^3) , \quad y(t) = y_0 + \frac{y_0}{2} \left( \frac{y_0^2}{2} - 3  \right) t^2 + O(t^4) .$$
\end{proposition}

\noindent If we set $y=0$, which corresponds to taking $y_0=0$ in Proposition \ref{prop:LocalExistenceBS_1}, we can again integrate the ODE \eqref{eq:dotxBS} (or equivalently \eqref{eq:dotxBS2}) and obtain:
$$x(t)= \frac{A_1(t)}{\frac{1}{2}(B_1^2(t)-\frac{1}{3})}.$$
From Proposition \ref{prop:LocalExistenceBS_1} we see that the corresponding instanton extends smoothly over $S^3$ on $P_{\id}$, and hence we find
 another $\GG_2$-instanton on the Bryant--Salamon $\mathbb{R}^4\times S^3$.

\begin{theorem}\label{thm:Alim}
The $\GG_2$-instanton $A^{\lim}$ arising from the case when $y_0=0$ in Proposition \ref{prop:LocalExistenceBS_1} is given by
$$A^{\lim}=\frac{2(r^3-1)}{3r(r^2-1)}\left(\sum_{i=1}^3T_i\otimes\eta_i^+\right).$$ 
Moreover, $A^{\lim}$ extends as a $\SU(2)^3$-invariant $\GG_2$-instanton to the Bryant--Salamon   
 $\mathbb{R}^4\times S^3 $.
\end{theorem}

\noindent It is straightforward to compute the curvature of $A^{x_1}$ and $A^{lim}$ and see that they decay at infinity but that their curvatures do not lie in $L^2$.

\subsubsection{The moduli space}

We have seen from Theorem \ref{thm:Clarke} that we have a moduli space $\mathcal{M}^{BS}_{P_1}$ of irreducible $\SU(2)^2\times\U(1)$-invariant $\GG_2$-instantons on $P_1$ which is parameterized by $x_1\in(0,+\infty)$.  Therefore, this moduli space is clearly non-compact.  A natural question is whether it can be compactified and, if so, what the compactification is: it is clear what happens at $x_1=0$, since we just take the trivial flat connections, but we need to understand what happens as $x_1\to+\infty$.  In \cite{Lotay2018} it is shown that $\mathcal{M}_{P_1}^{BS}$ can be compactified to the closed interval: it is demonstrated that $A^{\lim}$ is, in a certain precise sense, the limit of the $A^{x_1}$ as $x_1\to+\infty$.  The result, stated below, confirms expectations from \cites{Tian2000,Tao2004}.\\ To state the result we now introduce some notation for the re-scaling we wish to perform: for $p \in S^3$ and $\delta >0$ we define the map $s^p_{\delta}$ from the unit ball $B_1\subseteq \mathbb{R}^4$ by
$$s^p_{\delta} : B_1\subseteq \mathbb{R}^4 \rightarrow B_{\delta}\times\{p\}\subseteq \mathbb{R}^4\times S^3, \ \ x \mapsto (\delta x,p).$$
Recall that if we view $\mathbb{R}^4\setminus\{0\}=\mathbb{R}^+_t\times S^3$ then the basic ASD instanton on $\mathbb{R}^4$ with scale $\lambda>0$ can be written as
\begin{equation}\label{eq:basic.ASD}
A^{\text{ASD}}_{\lambda}= \frac{\lambda t^2}{1+\lambda t^2} \sum_{i=1}^3T_i\otimes\eta_i^+ .
\end{equation}

\begin{theorem}\label{thm:Compactness}
Let $\lbrace A^{x_1} \rbrace$ be a sequence of Clarke's $\GG_2$-instantons from Theorem \ref{thm:Clarke} with $x_1 \rightarrow + \infty$. 
\begin{itemize}
\item[(a)]  After a suitable rescaling, the family $\lbrace A^{x_1} \rbrace$ bubbles off a basic anti-self-dual instanton transversely to the associative $ S^3 = \lbrace 0 \rbrace \times S^3$.\\
More precisely, given any $\lambda>0$, there is a sequence of positive real numbers $\delta=\delta(x_1,\lambda) \rightarrow 0$ as $x_{1} \rightarrow + \infty$ such that: for all $p \in S^3$, $(s^p_{\delta})^* A^{x_1}$ converges uniformly with all derivatives to the basic ASD instanton $A^{\text{\emph{ASD}}}_{\lambda}$ on $B_1\subseteq\mathbb{R}^4$ as in \eqref{eq:basic.ASD}. 
\item[(b)] The connections $A^{x_1}$ converge  uniformly with all derivatives to $A^{\lim}$, given in Theorem \ref{thm:Alim}, on every compact subset of $(\mathbb{R}^4 \setminus \{0\})\times S^3  $ as $x_1\to +\infty$. 
\item[(c)] The function $\vert F_{A^{x_1}} \vert^2 - \vert F_{A_{\lim}} \vert^2$ is integrable for all $x_1>0$. Moreover, as $x_1 \rightarrow + \infty$ it converges to $8 \pi^2 \delta_{\lbrace 0 \rbrace \times S^3}$ as a current, i.e.~for all compactly supported functions $f$ we have
$$\lim_{x_1 \rightarrow + \infty}\int_{\mathbb{R}^4 \times S^3} f  (\vert F_{A^{x_1}} \vert^2 - \vert F_{A_{\lim}} \vert^2) \ \dvol_g = 8 \pi^2 \int_{\lbrace 0 \rbrace \times S^3} f \  \dvol_{g \vert_{\lbrace 0 \rbrace \times S^3}}.$$
\end{itemize}
\end{theorem}

\noindent Whilst (a) gives the familiar ``bubbling'' behaviour of sequences of instantons, with curvature concentrating on an associative $S^3$ by (c), we can interpret (b) as a ``removable singularity'' phenomenon since $A^{\lim}$ is a smooth connection on $\mathbb{R}^4\times S^3$.   In proving Theorem \ref{thm:Compactness}, we show that as $\lbrace A^{x_1} \rbrace$ bubbles along the associative $S^3$ one obtains a Fueter section, as in \cites{Donaldson2009, Haydys2011, Walpuski2017}.  Here this is just a constant map  from $S^3$ to the moduli space of anti-self dual connections on $\mathbb{R}^4$ (thought of as a fibre of the normal bundle), taking value at the basic instanton on $\mathbb{R}^4$.  Since $8\pi^2$ is the Yang--Mills energy of the basic instanton, we can also view (c) as the expected  ``conservation of energy''.\\
It is also worth observing that all of the $\GG_2$-instantons $A^{x_1}$ for $x_1>0$ and $A^{\lim}$ are asymptotic to the canonical pseudo-Hermitian--Yang--Mills connection on the standard nearly K\"ahler $S^3\times S^3$ given by:
\begin{equation}\label{eq:ainfty}
a_{\infty}=\frac{2}{3}\sum_{i=1}^3 T_i\otimes\eta_i^+. 
\end{equation}

\begin{proposition}\label{prop:asym}
Let $a_{\infty}$ be the canonical pseudo-Hermitian--Yang--Mills connection on $S^3\times S^3$ given in \eqref{eq:ainfty}.
\begin{itemize}
\item[(a)] If 
$A=A^{x_1}$ for some $x_1 \in \mathbb{R}^+$, then for $t \gg 1$
$$\vert A^{x_1} - a_{\infty} \vert \leq \frac{c}{ x_1  t^3},$$
where $c>0$ is some constant independent of $x_1$;
\item[(b)] If $A=A^{\lim}$, then for $t \gg 1$, $\vert A^{\lim} - a_{\infty} \vert = O(t^{-4})$. 
\end{itemize}
\end{proposition}

\subsection{Open problems}\label{ss:AC.probs}

There are several natural open problems which arise for $\GG_2$-instantons in the asymptotically conical setting.
\begin{itemize}
\item[(a)] Recently, infinitely many new examples of AC $\GG_2$-manifolds $X$ have been found \cite{FHN2018}.  These examples are cohomogeneity-1 for an action of $\SU(2)\times\SU(2)\times\U(1)$ and so the theory of $\GG_2$-instantons with these symmetries developed in \cite{Lotay2018} applies. Therefore, the local existence of $\GG_2$-instantons near the singular orbit in $X$ is guaranteed, and the open question is how many of these local solutions extend globally on $X$. Once one has classified the global solutions and has a non-trivial family, one can then ask about global properties of the moduli space of solutions, such as those discussed above. This would be specially interesting for the family $\mathbb{D}_7$ of \cite{FHN2018}, as this was not considered in \cite{Lotay2018}.
\item[(b)] As we have seen, $\GG_2$-instantons on AC $\GG_2$-manifolds naturally have limits at infinity which are pseudo-Hermitian--Yang--Mills (also known as nearly K\"ahler instantons) on the nearly K\"ahler link of the asymptotic cone at infinity.  There are very few examples of such connections on nearly K\"ahler 6-manifolds, and it is an important open problem to try to construct some examples on the known nearly K\"ahler 6-manifolds which arise as links of asymptotic cones of AC $\GG_2$-manifolds: that is, $S^6$, $\mathbb{C}\mathbb{P}^3$, $S^3\times S^3$ (and finite quotients thereof) and the flag $\mathbb{F}_2$.  Given these examples of nearly K\"ahler instantons, one can then ask if they arise as limits of $\GG_2$-instantons on AC $\GG_2$-manifolds.  If they do arise, it is then natural to ask how many $\GG_2$-instantons have the given nearly K\"ahler instanton as their limits at infinity.  
\item[(c)] An obvious problem in this context is to understand the local geometry of the moduli space of $\GG_2$-instantons on AC $\GG_2$-manifolds; i.e.~the deformation theory of such $\GG_2$-instantons.  This is currently being investigated by Joe Driscoll (a PhD student of Derek Harland) and would potentially help solve several interesting questions.  For example, can one prove a uniqueness result for the ``basic'' $\GG_2$-instanton on $\mathbb{R}^7$ \cite{Gunaydin95} (which has gauge group $\GG_2$)?  Do deformations of $\GG_2$-instantons with symmetries on AC $\GG_2$-manifolds also have symmetries?  A positive answer to the latter question would mean that we could describe (at least a component) of the moduli space of $\GG_2$-instantons on AC $\GG_2$-manifolds with a cohomogeneity-1 action via the techniques and results described in this survey.  There will also be a natural projection map in this context from the moduli space of $\GG_2$-instantons to the moduli space of nearly K\"ahler instantons (studied in \cite{Charbonneau2016}), and so it would be interesting to understand the properties of this map, e.g. whether it is surjective.
\item[(d)] What is the limit as $c\rightarrow + \infty$ of the $G_2$-instantons in Theorem \ref{SU(3)instantons}?
\item[(e)] We have seen that the local $\GG_2$-instanton defined on $P_{\id}$ given by Proposition \ref{prop:LocalExistenceBS_1} for $y_0=0$ extends globally to the Bryant--Salamon $\mathbb{R}^4\times S^3$ by Theorem \ref{thm:Alim}.  A concrete question is  whether any of the other local $\GG_2$-instantons from Theorem \ref{SU(3)instantons}. \ref{prop:LocalExistenceBS_1} for $y_0\neq 0$ extend globally or not.  Some numerical investigation suggests that if they do, their curvature is unbounded at infinity.
\end{itemize}

\section{Asymptotically locally conical (ALC) \texorpdfstring{$\GG_2$-manifolds}{G2-manifolds}}

A noncompact $\GG_2$-manifold  is said to be asymptotically locally conical (ALC), if it is asymptotic (at infinity) to a circle bundle over a 6-dimensional cone. The central part of this section is to summarize the results of \cite{Lotay2018}, where the authors' studied $\GG_2$-instantons on the so-called BGGG $\GG_2$-manifold: this is an ALC holonomy $\GG_2$-metric on $\mathbb{R}^4\times S^3$ constructed in \cite{Brandhuber2001}, and coming in a $1$-parameter family of torsion-free $\GG_2$-structures found a posteriori in \cite{Bogo2013}. In fact, the authors construction of instantons on the BGGG extends to give instantons for any holonomy $\GG_2$-metric in this whole $1$-parameter family, see Remark 13 in \cite{Lotay2018}.\\
This section is organized as follows. In $\S$\ref{ss:ALC.structure} we present some general structure results on ALC $\GG_2$-manifolds, for example we describe the induced structure on the asymptotic 
 circle bundle over a cone, since this asymptotic geometry is less familiar. Then, in $\S$\ref{ss:G2_Instantons_ALC} we characterise the limits of $\GG_2$-instantons with pointwise decaying curvature at infinity. Finally, in $\S$\ref{ss:BGGG_Instantons} we summarize the results of \cite{Lotay2018} and present some open problems in $\S$\ref{ss:ALC.probs}.

\subsection[\texorpdfstring{The $\GG_2$-structure}{The G2-structure}]{The {\boldmath $\GG_2$}-structure}\label{ss:ALC.structure}

A noncompact $\GG_2$-manifold $(X,\varphi)$ is said to be ALC if there is:
\begin{itemize}
\item a $\U(1)$-bundle $\pi: \Sigma^6 \rightarrow \Gamma^5$ and a $\U(1)$-invariant $\GG_2$-structure $\varphi_{\infty}$ on $(1,+ \infty) \times \Sigma$, whose associated metric is
$$g_{\varphi_{\infty}}=dr^2 + m^{2} \eta_{\infty}^2 + r^2 \pi^* g_5,$$
where $m \in \mathbb{R}^+$, $\eta_{\infty}$ is a connection on $\Sigma$ and $g_5$ a metric on $\Gamma$;
\item a compact set $K \subset X$ and (up to a double cover)\footnote{The possible need for the double cover is because $X$ may only be asymptotic to an $S^1$-bundle, but we can get a principal bundle by taking a double cover.} 
a diffeomorphism $p: (1,+ \infty)_r \times \Sigma \rightarrow X \backslash K$,
\end{itemize}
such that if $\nabla$ denotes the Levi-Civita connection of $g_{\varphi_{\infty}}$ then
\begin{equation}\label{eq:ALC.decay}
\vert \nabla^j ( \varphi_{\infty} - p^* \varphi \vert_{X \backslash K} 
)\vert_{g_{\varphi_{\infty}}} = O(r^{\nu-j}) \quad\text{ as }r\to+\infty,
\end{equation}
for some $\nu<0$ and $j=0,1$.  
\\ 
 Our next result describes the structure on $(1,+\infty)\times\Sigma$ 
 induced from the torsion-free $\GG_2$-structure $\varphi$ on $X$ and limits the range of rates $\nu$ to consider.

\begin{proposition}\label{prop:SU(2)StructureLimit}
Let $(X,\varphi)$ be an ALC $\GG_2$-manifold and use the notation above.
\begin{itemize}
\item[(a)]  If $\nu<0$, 
the metric $g_5$ is induced by a Sasaki--Einstein $SU(2)$-structure on $\Gamma$ given by $(\alpha, \omega_1,\omega_2,\omega_3)$ satisfying 
\end{itemize}

\vspace{-24pt}

\begin{equation}\label{eq:su2structure}
d\alpha=-2\omega_1, \quad d\omega_2 = 3\alpha \wedge \omega_3, \quad d \omega_3 = -3 \alpha \wedge \omega_2.
\end{equation}
\begin{itemize}
\item[] Hence, the cone metric $dr^2+r^2g_5$ on $(1,+\infty)_r\times \Gamma$ is Calabi--Yau.
\item[(b)] If $\nu<-1$, then $d\eta_{\infty}=0$, and thus the connection is flat.
\end{itemize}
\end{proposition}

\noindent Now we know from Proposition \ref{prop:SU(2)StructureLimit} that the asymptotic cone for an ALC $\GG_2$-manifold is Calabi--Yau, we can impose a further condition on the connection $\eta_{\infty}$ for the definition of an ALC $\GG_2$-manifold: namely, that $\eta_{\infty}$ is Hermitian--Yang--Mills, i.e.~$d\eta_{\infty}\wedge\omega^2=0$ and $d\eta_{\infty}\wedge\Omega_2=0$.\\
We now give the example of the standard Sasaki-Einstein structure on $S^2 \times S^3$ in terms of the framework above. We shall see that is the most important for our study.\\
Let $S^2 \times S^3= \SU(2)^2/ \Delta \U(1)$ and let $\lbrace \eta_i^+, \eta_i^- \rbrace_{i=1}^3$ be as in \S\ref{ss:BS.R4xS3}.
We can equip $S^3 \times S^3 \rightarrow S^2 \times S^3$ with a connection such that $\eta_2^+,\eta_3^+,\eta_1^-,\eta_2^-,\eta_3^-$ is a horizontal coframing. We define:
\begin{gather*}
\eta_{\infty}=2\eta_1^+,\quad \alpha  =  -\frac{4}{3} \eta_1^-, \quad \omega_1  =  \frac{4}{3} \left( \eta_2^+ \wedge \eta_3^- + \eta_2^- \wedge \eta_3^+ \right),  \\
 \quad
 \omega_2  =  \frac{4}{3} \left( \eta_{2}^+\wedge\eta_{3}^+ - \eta_{2}^-\wedge\eta_{3}^- \right),\quad
\omega_3  =  \frac{4}{3} \left( \eta_2^+ \wedge \eta_2^- + \eta_3^+ \wedge \eta_3^- \right).
\end{gather*}
The forms $\alpha,\omega_1,\omega_2,\omega_3$ are basic for the $\Delta \U(1)$-action and equip $S^2\times S^3$ with an $\SU(2)$-structure.  We can check that \eqref{eq:su2structure} holds 
 and so this is the standard homogeneous Sasaki--Einstein structure on $S^2 \times S^3$.   The conical Calabi--Yau metric arising from this Sasaki--Einstein structure on $S^2\times S^3$ is known as the conifold or 3-dimensional ordinary double point.\\
We also see that $\eta_{\infty}$ is a connection form on $S^3\times S^3$ such that 
$$d \eta_{\infty} = -4 \left( \eta_{2}^+\wedge\eta_3^+ + \eta_{2}^-\wedge\eta_{3}^- \right)$$
is basic anti-self-dual: i.e.~$d \eta_{\infty} \wedge \omega_i=0$  for $i=1,2,3$.  This implies that $\eta_{\infty}$ is Hermitian--Yang--Mills.

\subsection[\texorpdfstring{$\GG_2$-instantons}{G2-instantons}]{{\boldmath $\GG_2$}-instantons}\label{ss:G2_Instantons_ALC}

We now study the asymptotic behaviour of $\GG_2$-instantons on ALC $\GG_2$-manifolds, and begin by examining the $\GG_2$-instanton condition on the asymptotic $\U(1)$-bundle over a Calabi--Yau cone.  We shall use the notation of the previous subsection.\\
Let $\pi: (1,+\infty)_r \times \Sigma \rightarrow (1,+ \infty)_r \times \Gamma$ be a $\U(1)$-bundle over a Calabi--Yau cone, equipped with the $\GG_2$-structure
$$\varphi_{\infty} = m \eta_{\infty} \wedge \omega + \Omega_1,$$
as above. Let $P$ be the pullback to $(1,+\infty)\times \Sigma$ of a bundle over $M$. If $A_{\infty}$ is a connection on $P$, then without loss of generality we can write it as
\begin{equation}\label{eq:ALC.limit.A}
A_{\infty}=a+m\Phi \otimes \eta_{\infty},
\end{equation}
for a connection $a$ pulled back from $(1,+\infty)\times\Gamma$ and $\Phi\in\Omega^0((1,+\infty)\times\Sigma,\mathfrak{g}_P)$.\\
In our case we will be investigating $\GG_2$-instantons  that are invariant under the $\U(1)$-action on the end of $X$; that is, we take a lift of the $\U(1)$-action to the total space and the connection is invariant under the lifted action.  If we assume $\eta_{\infty}$ is Hermitian--Yang--Mills, then the  conditions for a $\U(1)$-invariant connection $A_{\infty}$ as in \eqref{eq:ALC.limit.A} to be a $\GG_2$-instanton are then
\begin{equation}\label{eq:PreMonopole2}
 F_a \wedge \Omega_2 = - \frac{1}{2 }d_a \Phi  \wedge \omega^2,
 \quad  F_a \wedge \omega^2 =0.
\end{equation}
These are the equations for a Calabi--Yau monopole $(a, \Phi)$ on $(1,+\infty) \times \Gamma$ equipped with the conical torsion-free $SU(3)$-structure $(\omega, \Omega_2)$.\\
These observations lead to the following.

\begin{proposition}\label{prop:ALC.instant}
Let $A$ be a $\GG_2$-instanton on an ALC $\GG_2$-manifold $(X,\varphi)$ and use the notation from the start of $\S$\ref{ss:ALC.structure}. Suppose there exists a $\U(1)$-invariant connection $A_{\infty}=a+m\Phi\otimes\eta_{\infty}$ as in \eqref{eq:ALC.limit.A} such that $p^*F_A|_{X\setminus K}$ is asymptotic at infinity to $F_{A_{\infty}}$.
  Then $(a,\Phi)$ is a Calabi--Yau monopole on the Calabi--Yau cone $(1,+\infty)\times \Gamma$.
\end{proposition}

\subsection{On the BGGG-Bogoyavlenskaya \texorpdfstring{$\mathbb{R}^4 \times S^3$}{R4xS3}}\label{ss:BGGG_Instantons}

On $\mathbb{R}^4 \times S^3$, as well as the Bryant--Salamon metric, there is another explicit complete $\GG_2$-holonomy metric constructed by Brandhuber and collaborators in \cite{Brandhuber2001}, which we will abbreviate to BGGG.  The BGGG metric is a member of a family of complete $\SU(2)^2\times \U(1)$-invariant, cohomogeneity-1,  $\GG_2$-holonomy metrics on $\mathbb{R}^4\times S^3$ found in \cite{Bogo2013}.

\subsubsection{The BGGG metric}

\noindent To derive the BGGG example, we return to the setting of $\SU(2)^2\times\U(1)$-symmetry in $\S$\ref{ss:symmetry}: in particular we recall the functions $A_1,A_2,B_1,B_2$ defining the metric, satisfying \eqref{eq:dotA1}--\eqref{eq:dotB2}.  One can choose $c>0$, set $B_1=s$ and
$$A_1=c\frac{ds}{dt}=c\frac{A_2^2+B_2^2-s^2}{A_2B_2}$$
from \eqref{eq:dotB1}.  Letting $C_{\pm}=A_2^2\pm B_2^2$ the equations 
\eqref{eq:dotA2} and \eqref{eq:dotB2} yield
$$\frac{d}{ds}C_+=\frac{s^2C_+-C_-^2}{s(C_+-s^2)}\quad\text{and}\quad 
\frac{d}{ds}C_-= - \frac{C_-}{s}-2c.$$
The second equation is easily integrated and so we are able to find solutions 
$$C_+(s)=\frac{3s^2-c^2}{2}\quad\text{and}\quad C_-(s)=-cs.$$
We thus obtain a one-parameter family of solutions to \eqref{eq:dotA1}-\eqref{eq:dotB2}:
\begin{gather}\label{eq:BGGG1} A_1(s)=2c\sqrt{\frac{s^2-c^2}{9s^2-c^2}},\quad A_2(s)=\frac{1}{2}\sqrt{(3s+c)(s-c)},\\
\label{eq:BGGG2}
B_1(s)=s,\quad B_2(s)=\frac{1}{2}\sqrt{(3s-c)(s+c)},
\end{gather}
defined for $s\geq c>0$. 
These solutions give holonomy $\GG_2$ metrics on $\mathbb{R}^4\times S^3$.  
We can further scale the metric from $g$ to $\lambda^2 g$ and the resulting fields scale as $A_i^{\lambda}(s)=\lambda A_i (s/\lambda)$, $B_i^{\lambda}(s)=\lambda B_i(s/\lambda)$. These give the following family of solution to the ODEs \eqref{eq:dotA1}-\eqref{eq:dotB2} above:
\begin{gather}\nonumber A_1^{\lambda}(s)= 2c \lambda \sqrt{\frac{s^2-c^2\lambda^2}{9s^2-c^2\lambda^2}},\quad  A_2^{\lambda}(s)=\frac{1}{2}\sqrt{(3s+c\lambda)(s-c \lambda)},\\
\nonumber
B_1^{\lambda}(s)= s , \quad B_2^{\lambda}(s)=\frac{1}{2}\sqrt{(3s-c\lambda)(s+c\lambda)}.
\end{gather}
We see that under the scaling we have $c\mapsto c\lambda$, so we can always scale so that $c=1$.  In particular, one can set $\lambda=3/2$, $c=1$ and as in \cite{Brandhuber2001} define the coordinate $r\in [9/4,+\infty)$ implicitly by 
\begin{equation}\label{eq:r.BGGG}
t(r)=\int_{9/4}^{r} \frac{\sqrt{(s-3/4)(s+3/4)}}{\sqrt{(s-9/4)(s+9/4)}} ds
\end{equation} and find that
\begin{gather*}
A_1 =\frac{\sqrt{(r-9/4)(r+9/4)}}{\sqrt{(r-3/4)(r+3/4)}}, \quad A_2=A_3=\sqrt{\frac{(r-9/4)(r+3/4)}{3}},\\
B_1=\frac{2r}{3}, \quad B_2=B_3=\sqrt{\frac{(r-3/4)(r+9/4)}{3}}
\end{gather*}
solve \eqref{eq:dotA1}-\eqref{eq:dotB2}.  
We see in this BGGG case that the principal orbits are again $S^3\times S^3$ and the singular orbit $\lbrace 0 \rbrace \times S^3$ is associative. \\
It is straightforward to see that the BGGG is ALC with rate $\nu=-1$ and $m=1$: in fact, the metric is asymptotic to
$$h=dt^2 + 4(\eta_1^+ )^2 + \frac{4t^2}{3} \left( (\eta_2^+ )^2 + ( \eta_3^+)^2 \right)+ 
\frac{16t^2}{9}(\eta_1^- )^2 +  \frac{4t^2}{3} \left( (\eta_2^- )^2 + (\eta_3^-)^2 \right),$$
which is a circle bundle over the Calabi--Yau cone over the standard homogeneous Sasaki--Einstein structure on $S^2\times S^3$ described in $\S$\ref{ss:ALC.structure}.    This is in particular shows that 
Proposition \ref{prop:SU(2)StructureLimit}(b) is sharp.\\
In \cite{Bogo2013}, Bogoyavlenskaya constructed a $1$-parameter family (up to scaling) of $\SU(2)^2\times \U(1)$-invariant, cohomogeneity-1, $\GG_2$-holonomy metrics on $\mathbb{R}^4\times S^3$, obtained by continuously deforming the BGGG metric. With these metrics, one can independently vary the size of the circle at infinity and the associative $S^3$, and thus, in particular, obtain the BS metric as a limit of the family.  The BGGG metric is the only one from \cite{Bogo2013} which is explicitly known. 

\subsubsection{Examples of \texorpdfstring{$\GG_2$-instantons}{G2-instantons}}

It is again straightforward to write down the evolution equation \eqref{eq:IEvolution} for $\SU(2)^2$-invariant $\GG_2$-instantons on a $\U(1)$-bundle over the Bogoyavlenskaya metrics on $\mathbb{R}^4\times S^3$.  One can solve this equation explicitly in the BGGG case and obtain the following result.

\begin{proposition}\label{prop:BGGG.U1}
Any $\SU(2)^2$-invariant $\GG_2$-instanton $A$ with gauge group $\U(1)$ over the BGGG $\mathbb{R}^4\times S^3$ can be written as  
\begin{equation*}
A=\frac{(r-9/4)(r+9/4)}{(r-3/4)(r+3/4)}x_1\eta_1^++
\frac{(r-9/4)e^r}{\sqrt{r}(r+9/4)^2}(x_2\eta_2^++x_3\eta_3^+)
\end{equation*}
for some $x_1,x_2,x_3\in\mathbb{R}$, where $r\in[9/4,+\infty)$ is given by \eqref{eq:r.BGGG}.   When $x_2=x_3=0$, $A$ is a multiple of the harmonic 1-form dual to the Killing field generating the $\U(1)$-action.
\end{proposition}

\noindent We already observe a marked difference in 
the behaviour of $\GG_2$-instantons on the BS and BGGG $\mathbb{R}^4\times S^3$ in this simple abelian setting.  In particular, the instantons in the BS case all have bounded curvature, whereas those in the BGGG case have bounded curvature only when $x_2=x_3=0$, in which case the curvature also decays to $0$ as $r\to\infty$.\\
We now turn to gauge group $\SU(2)$ and begin by simplifying the ODEs \eqref{eq:IEvolution} in the $\SU(2)^2\times\U(1)$-invariant setting.

\begin{proposition}\label{prop:ODEsSU(2)xU(2)}
Let $A$ be an $\SU(2)^2 \times \U(1)$-invariant  $\GG_2$-instanton on 
$\mathbb{R}^+\times \SU(2)^2\cong \mathbb{R}^+\times (\SU(2)^2\times \U(1)/\Delta \U(1))$ with gauge group $\SU(2)$. 
There is a standard basis  $\lbrace T_i \rbrace_{i=1}^3$ of $\mathfrak{su}(2)$, i.e.~with $[T_i,T_j]=2\epsilon_{ijk} T_k$, such that (up to an invariant gauge transformation) we can write
\begin{eqnarray}\label{eq:InvariantConnection.BGGG}
A & = &  A_1 f^{+} T_1\otimes \eta_1^{+}  + A_2g^{+} (T_2\otimes\eta_2^{+}  + T_3\otimes\eta_3^{+} ) \\ \nonumber
& &
 + B_1 f^{-} T_1\otimes \eta_1^{-} + B_2 g^{-} (T_2\otimes\eta_2^{-}  + T_3\otimes\eta_3^{-}),
\end{eqnarray}
with $f^{\pm}$, $g^{\pm}: \mathbb{R}^+ \rightarrow \mathbb{R}$ satisfying 
\begin{align}\label{eq:dotf+}
\dot{f}^+ + 
\frac{1}{2}\left(\frac{A_1}{B_2^2}-\frac{A_1}{A_2^2}\right)f^+ & =  (g^-)^2 - (g^+)^2,\\ 
\label{eq:dotg+}
\dot{g}^+ +  
\frac{1}{2}\left(\frac{A_2^2+B_1^2+B_2^2}{A_2B_1B_2}-\frac{A_1^2+2A_2^2}{A_1A_2^2}\right)g^+ & =  f^-g^- - f^+ g^+ ,\\
\label{eq:dotf-}
\dot{f}^- +  
\left(\frac{A_2^2+B_1^2+B_2^2}{A_2B_1B_2}\right) f^- & =  2g^- g^+, \\ 
\label{eq:dotg-}
\dot{g}^- + 
\frac{1}{2}\left(\frac{A_2^2+B_1^2+B_2^2}{A_2B_1B_2}+\frac{A_1^2+2B_2^2}{A_1B_2^2}\right)g^- & =  g^- f^+ + g^+ f^- .
\end{align}
\end{proposition}

\noindent We can then determine the local conditions for these connections to extend over the singular orbit.

\begin{lemma}\label{lem:SmoothlyExtendBGGG}
The connection $A$ in \eqref{eq:InvariantConnection.BGGG} extends smoothly over the singular orbit $S^3$ if and only if  $f^+$ and $g^+$ are odd, $f^-$ and $g^-$ are even, and their Taylor expansions around  $t=0$ are:
\begin{itemize}\item either
\begin{align*}
f^{-}&= 
f_2^- t^2 + O(t^4), &  g^-&= 
g_2^- t^2 + O(t^4),\\ \nonumber
f^{+}&= f_1^+ t + O(t^3), &  g^+&= g_1^+ t + O(t^3),
\end{align*}
in which case $A$ extends smoothly as a connection on $P_{1}$; 
\item or
\begin{align*}
f^{-}&= b_0^- + b_2^- t^2 + O(t^4) , &  g^-&= b_0^- + b_2^- t^2 + O(t^4), \\ \nonumber
f^{+}&= \frac{2}{t} + (b_2^+-\frac{2}{3}\dddot{A_1}(0) ) t + O(t^3) , &   g^+&= \frac{2}{t} + (b_2^+-\frac{2}{3}\dddot{A_2}(0)) t + O(t^3),
\end{align*}
in which case $A$ extends smoothly as a connection on $P_{\id}$.
\end{itemize}
\end{lemma}

\noindent We can now answer the question of how many $\GG_2$-instantons there are defined near the singular orbit on any $\SU(2)^2\times\U(1)$-invariant $\GG_2$-manifold, which extend smoothly on $P_1$. 

\begin{proposition}\label{prop:LocalG2Instantons}
Let $X \subset \mathbb{R}^4\times S^3 $ contain the singular orbit $\{0\}\times S^3 $ of the $\SU(2)^2 \times \U(1)$ action and be equipped with an $\SU(2)^2 \times \U(1)$-invariant holonomy $\GG_2$-metric.  There is a $2$-parameter family of $\ \SU(2)^2 \times \U(1)$-invariant $\GG_2$-instantons $A$ with gauge group $\SU(2)$ in a neighbourhood of the singular orbit in $X$  smoothly extending on $P_{1}$.
\end{proposition}

\noindent The BS, BGGG and Bogoyavlenskaya  
$\GG_2$-metrics all have $\SU(2)^2 \times \U(1)$-symmetry and so Proposition \ref{prop:LocalG2Instantons} yields a 2-parameter family of $\GG_2$-instantons in these cases.  In the BS case, we already stated in Theorem \ref{thm:Clarke} that only a 1-parameter family extends globally.  In contrast, we see in the BGGG case that there is a 2-parameter family of local $\GG_2$-instantons which extends globally with bounded curvature and another 2-parameter family which cannot be extended so as to have bounded curvature.

\begin{theorem}\label{thm:BGGG.NoInst}\label{thm:BGGG.Inst}
The moduli space $\mathcal{M}^{BGGG}_{P_1}$ 
of irreducible $\SU(2)^2 \times \U(1)$-invariant $\GG_2$-instantons with gauge group $\SU(2)$ on the BGGG metric, smoothly extending on $P_1$, contains a nonempty (and unbounded) open set which is parametrised by $U \subset \mathbb{R}^2$.\\
Specifically, let $A$ be a $\SU(2)^2\times \U(1)$-invariant $\GG_2$-instanton with gauge group $\SU(2)$ defined in a neighbourhood of $\{0\}\times S^3$ on the BGGG $\mathbb{R}^4\times S^3$  smoothly extending over $P_1$ as given by Proposition \ref{prop:LocalG2Instantons}.  
\begin{itemize}
\item[(a)] If $f_1^+\leq\frac{1}{2}$, or $g_1^+\geq 0$ with $g_1^+\geq f_1^+$, then $A$ extends globally to $\mathbb{R}^4\times S^3$ with bounded curvature if and only if $A$ has gauge group $\U(1)$ and is given in Proposition \ref{prop:BGGG.U1} with $x_2=x_3=0$; i.e.~we must have $g_1^+=0$ and for some $x_1\in\mathbb{R}$, 
$$A=\frac{(r-9/4)(r+9/4)}{(r-3/4)(r+3/4)}x_1\eta_1^+.$$
\item[(b)] If $f_1^+\geq \frac{1}{2}+g_1^+>\frac{1}{2}$, then $A$ is irreducible and extends globally to $\mathbb{R}^4\times S^3$ with bounded curvature. 
\end{itemize}
\end{theorem}

\noindent We also have the following interesting observations.

\begin{theorem}\label{th:NoG2Inst}
In the setting of Theorem \ref{thm:BGGG.NoInst} the following holds.
\begin{itemize}
\item[(a)] The instantons parametrised by $U$ have quadratically decaying curvature. 
\item[(b)] The map $Hol_{\infty} : U \rightarrow \U(1) \subset \SU(2)$, which evaluates the holonomy of the $\GG_2$-instanton along the finite size circle at $+\infty$, is surjective.
\end{itemize}
\end{theorem}

\noindent We can now also ask about $P_{\id}$ and we see even the local existence theory is different to the $P_1$ case.

\begin{proposition}\label{prop:LocalG2Instantons2}
Let $X \subset \mathbb{R}^4\times S^3 $ contain the singular orbit $\{0\}\times S^3 $ of the $SU(2)^2 \times U(1)$ action and be equipped with an $SU(2)^2 \times U(1)$-invariant holonomy $G_2$-metric.  There is a $1$-parameter family of $SU(2)^2 \times U(1)$-invariant $G_2$-instantons $A$ with gauge group $SU(2)$ in a neighbourhood of the singular orbit in $X$ smoothly extending over $P_{\id}$.
\end{proposition}

\noindent In this setting, unfortunately, we cannot yet find any global solutions on the BGGG manifold which extend smoothly on $P_{\id}$.




\subsection{Open problems}\label{ss:ALC.probs}

There are several natural open problems which present themselves for $\GG_2$-instantons on ALC $\GG_2$-manifolds.

\begin{itemize}
\item[(a)] Proposition \ref{prop:ALC.instant} shows that the natural limits of $\GG_2$-instantons on ALC $\GG_2$-manifolds (if they exist) are Calabi--Yau monopoles on Calabi--Yau cones.  These observations further motivate the study of Calabi--Yau monopoles on cones or AC Calabi--Yau 3-folds.  See \cite{Oliveira2014Thesis} and \cite{Oli16} for some examples and results on Calabi--Yau monopoles in the AC and conical settings. 
\item[(b)] It has been shown that \cite{FHN2017} there are many ALC $\GG_2$-manifolds which are close to the degenerate Calabi--Yau cone limit, and the typically example will only be $\U(1)$-invariant.  It is therefore interesting to attempt to construct $\GG_2$-instantons on these ALC $\GG_2$-manifolds which are, in a sense, close to Calabi--Yau monopoles on the cone.  The authors of this article are actively pursuing this problem.
\item[(c)] As in the AC case, it would be good to have a deformation theory for $\GG_2$-instantons on ALC $\GG_2$-manifolds.  This would have obvious relations to the deformation theory of Calabi--Yau monopoles on AC Calabi--Yau 3-folds, which also needs to be developed.  In particular, one can ask about the image of the projection map from the moduli space of $\GG_2$-instantons on an ALC $\GG_2$-manifold to the space of Calabi--Yau monopoles on the Calabi--Yau cone which appears at the end of the ALC $\GG_2$-manifold.
\item[(d)] On the BGGG $\GG_2$-manifold $\mathbb{R}^4\times S^3$, we have shown non-existence for irreducible $\SU(2)^2\times \U(1)$-invariant $\GG_2$-instantons with gauge group $\SU(2)$ and bounded curvature for  $g_1^+>0$ and $f_1^+\leq\frac{1}{2}$ or $g_1^+\geq f_1^+$, and existence for $f_1^+\geq\frac{1}{2}+g_1^+>\frac{1}{2}$.  This currently leaves open the region where $0<f_1^+-\frac{1}{2}<g_1^+<f_1^+$, which should be investigated so as to describe the full moduli space $\mathcal{M}_{P_1}^{BGGG}$.  Some numerical investigation indicates that some of these initial conditions may lead to globally defined instantons with bounded curvature and some may not.  Some of the existence and non-existence results for instantons for the BGGG metric extend, with suitable modifications, to all of the Bogoyavlenskaya metrics, but some do not, so it would be good to address this gap.  
\item[(e)] An interesting problem is to investigate the behaviour of   $\GG_2$-instantons as the underlying metric is deformed. For instance,  we have $\GG_2$-instantons on all of the Bogoyavlenskaya $\GG_2$-manifolds, and we would want to analyse these instantons as the size of the circle at infinity gets very large or small. 
When it gets very large we expect them to resemble $\GG_2$-instantons for the Bryant--Salamon metric in $\mathbb{R}^4\times S^3$.  When it gets very small, there may be a relation with Calabi--Yau monopoles on the deformed conifold 
(as in \cite{Oli16} or problem (b) above).  
\item[(f)] Even if we can describe the moduli space $\mathcal{M}_{P_1}^{BGGG}$, it will be non-compact so, just as in the Bryant--Salamon case, we will want to compactify it.  It is therefore certainly an interesting problem to investigate the behaviour of the family of instantons from Theorem \ref{thm:BGGG.Inst} when one or both of $f_1^+$ and $g_1^+$ go to 
infinity.   We would expect bubbling phenomena as in the Bryant--Salamon case in Theorem \ref{thm:Compactness}, with possible relationship to the ASD instantons on Taub--NUT found in \cite{EtesiHausel}.  
The lack of an explicit formula for our instantons makes the bubbling analysis more difficult, but it should clearly be explored.
\end{itemize}




\begin{bibdiv}
	\begin{biblist}
		
		\bib{Bazaikin2013}{article}{
			author={Bazaikin, Ya.~V.},
			author={Bogoyavlenskaya, O.~A.},
			title={Complete {R}iemannian metrics with holonomy group {$G\sb 2$} on
				deformations of cones over {$S\sp 3\times S\sp 3$}},
			date={2013},
			ISSN={0001-4346},
			journal={Math. Notes},
			volume={93},
			number={5-6},
			pages={643\ndash 653},
			url={http://dx.doi.org/10.1134/S0001434613050015},
			note={Translation of Mat. Zametki {{\bf{93}}} (2013), no. 5, 645--657},
			review={\MR{3206014}},
		}
		
		\bib{Brandhuber2001}{article}{
			author={Brandhuber, A.},
			author={Gomis, J.},
			author={Gubser, S.~S.},
			author={Gukov, S.},
			title={Gauge theory at large {$N$} and new {$G\sb 2$} holonomy metrics},
			date={2001},
			ISSN={0550-3213},
			journal={Nuclear Phys. B},
			volume={611},
			number={1-3},
			pages={179\ndash 204},
			url={http://dx.doi.org/10.1016/S0550-3213(01)00340-6},
			review={\MR{1857379}},
		}
		
		\bib{Bogo2013}{article}{
			author={Bogoyavlenskaya, O.~A.},
			title={On a new family of complete ${G}_2$-holonomy {R}iemannian metrics
				on ${S}^3 \times \mathbb{R}^4$},
			date={2013},
			ISSN={1573-9260},
			journal={Sib. Math. J.},
			volume={54},
			number={3},
			pages={431\ndash 440},
			url={http://dx.doi.org/10.1134/S0037446613030075},
			note={Translation of Sibirsk. Mat. Zh. {\bf 54} (2013), no. 3,
				551--562},
			review={\MR{3112613}},
		}
		
		\bib{Bryant2010}{incollection}{
			author={Bryant, R.~L.},
			title={Non-embedding and non-extension results in special holonomy},
			date={2010},
			booktitle={The many facets of geometry},
			publisher={Oxford Univ. Press, Oxford},
			pages={346\ndash 367},
			url={http://dx.doi.org/10.1093/acprof:oso/9780199534920.003.0017},
			review={\MR{2681703}},
		}
		
		\bib{Bryant1989}{article}{
			author={Bryant, R.~L.},
			author={Salamon, S.~M.},
			title={On the construction of some complete metrics with exceptional
				holonomy},
			date={1989},
			ISSN={0012-7094},
			journal={Duke Math. J.},
			volume={58},
			number={3},
			pages={829\ndash 850},
			url={http://dx.doi.org/10.1215/S0012-7094-89-05839-0},
			review={\MR{1016448 (90i:53055)}},
		}
		
		\bib{Corrigan1983}{article}{
			author={Corrigan, E.},
			author={Devchand, C.},
			author={Fairlie, D.~B.},
			author={Nuyts, J.},
			title={First-order equations for gauge fields in spaces of dimension
				greater than four},
			date={1983},
			ISSN={0550-3213},
			journal={Nuclear Phys. B},
			volume={214},
			number={3},
			pages={452\ndash 464},
			url={http://dx.doi.org/10.1016/0550-3213(83)90244-4},
			review={\MR{698892 (84i:81058)}},
		}
		
		\bib{Charbonneau2016}{article}{
			author={Charbonneau, B.},
			author={Harland, D.},
			title={Deformations of nearly k{\"a}hler instantons},
			date={2016Dec},
			ISSN={1432-0916},
			journal={Communications in Mathematical Physics},
			volume={348},
			number={3},
			pages={959\ndash 990},
			url={https://doi.org/10.1007/s00220-016-2675-y},
		}
		
		\bib{Haskins}{article}{
			author={Corti, A.},
			author={Haskins, M.},
			author={Nordstr{\"o}m, J.},
			author={Pacini, T.},
			title={{${G}_2$}-manifolds and associative submanifolds via semi-{F}ano
				3-folds},
			date={2015},
			ISSN={0012-7094},
			journal={Duke Math. J.},
			volume={164},
			number={10},
			pages={1971\ndash 2092},
			url={http://dx.doi.org/10.1215/00127094-3120743},
			review={\MR{3369307}},
		}
		
		\bib{Clarke14}{article}{
			author={Clarke, A.},
			title={Instantons on the exceptional holonomy manifolds of {B}ryant and
				{S}alamon},
			date={2014},
			ISSN={0393-0440},
			journal={J. Geom. Phys.},
			volume={82},
			pages={84\ndash 97},
			url={http://dx.doi.org/10.1016/j.geomphys.2014.04.006},
			review={\MR{3206642}},
		}
		
		\bib{Donaldson2009}{incollection}{
			author={Donaldson, S.~K.},
			author={Segal, E.~P.},
			title={Gauge theory in higher dimensions, {II}},
			date={2011},
			booktitle={Surveys in differential geometry. {V}olume {XVI}. {G}eometry of
				special holonomy and related topics},
			series={Surv. Differ. Geom.},
			volume={16},
			publisher={Int. Press, Somerville, MA},
			pages={1\ndash 41},
			review={\MR{2893675}},
		}
		
		\bib{Donaldson1998}{incollection}{
			author={Donaldson, S.~K.},
			author={Thomas, R.~P.},
			title={Gauge theory in higher dimensions},
			date={1998},
			booktitle={The geometric universe ({O}xford, 1996)},
			publisher={Oxford Univ. Press},
			address={Oxford},
			pages={31\ndash 47},
			url={http://www.ma.ic.ac.uk/~rpwt/skd.pdf},
			review={\MR{MR1634503 (2000a:57085)}},
		}
		
		\bib{EtesiHausel}{article}{
			author={Etesi, G.},
			author={Hausel, T.},
			title={Geometric construction of new {Y}ang-{M}ills instantons over
				{T}aub-{NUT} space},
			date={2001},
			ISSN={0370-2693},
			journal={Phys. Lett. B},
			volume={514},
			number={1-2},
			pages={189\ndash 199},
			url={http://dx.doi.org/10.1016/S0370-2693(01)00821-8},
			review={\MR{1850138}},
		}
		
		\bib{FHN2017}{article}{
			author={{Foscolo}, L.},
			author={{Haskins}, M.},
			author={{Nordstr{\"o}m}, J.},
			title={{Complete non-compact ${G}_2$-manifolds from asymptotically
					conical Calabi--Yau 3-folds}},
			date={2017-09},
			journal={ArXiv e-prints},
			eprint={1709.04904},
		}
		
		\bib{FHN2018}{article}{
			author={{Foscolo}, L.},
			author={{Haskins}, M.},
			author={{Nordstr{\"o}m}, J.},
			title={{Infinitely many new families of complete cohomogeneity one
					${G}_2$-manifolds: ${G}_2$ analogues of the Taub-NUT and Eguchi--Hanson
					spaces}},
			date={2018-05},
			journal={ArXiv e-prints},
			eprint={1805.02612},
		}
		
		\bib{Gunaydin95}{article}{
			author={G{\"u}naydin, M.},
			author={Nicolai, H.},
			title={Seven-dimensional octonionic yang-mills instanton and its
				extension to an heterotic string soliton},
			date={1995},
			ISSN={0370-2693},
			journal={Physics Letters B},
			volume={351},
			number={1},
			pages={169 \ndash  172},
			url={http://www.sciencedirect.com/science/article/pii/037026939500375U},
		}
		
		\bib{Haydys2011}{article}{
			author={Haydys, A.},
			title={Gauge theory, calibrated geometry and harmonic spinors},
			date={2012},
			ISSN={0024-6107},
			journal={J. Lond. Math. Soc. (2)},
			volume={86},
			number={2},
			pages={482\ndash 498},
			url={http://dx.doi.org/10.1112/jlms/jds008},
			review={\MR{2980921}},
		}
		
		\bib{Haydys2015}{article}{
			author={Haydys, A.},
			author={Walpuski, T.},
			title={A compactness theorem for the {S}eiberg--{W}itten equation with
				multiple spinors in dimension three},
			date={2015},
			ISSN={1420-8970},
			journal={Geometric and Functional Analysis},
			volume={25},
			number={6},
			pages={1799\ndash 1821},
			url={http://dx.doi.org/10.1007/s00039-015-0346-3},
		}
		
		\bib{Joyce2000}{book}{
			author={Joyce, D.~D.},
			title={Compact manifolds with special holonomy},
			series={Oxford Mathematical Monographs},
			publisher={Oxford University Press},
			address={Oxford},
			date={2000},
			ISBN={0-19-850601-5},
			review={\MR{1787733 (2001k:53093)}},
		}
		
		\bib{Kovalev2003}{article}{
			author={Kovalev, A.},
			title={Twisted connected sums and special {R}iemannian holonomy},
			date={2003},
			ISSN={0075-4102},
			journal={J. Reine Angew. Math.},
			volume={565},
			pages={125\ndash 160},
			url={http://dx.doi.org/10.1515/crll.2003.097},
			review={\MR{MR2024648 (2004m:53088)}},
		}
		
		\bib{Lotay2018}{article}{
			author={Lotay, J.~D.},
			author={Oliveira, G.},
			title={${SU}(2)^2$-invariant ${G}_2$-instantons},
			date={2018},
			journal={Mathematische Annalen},
			volume={371},
			pages={961\ndash 1011},
		}
		
		\bib{Menet2015}{article}{
			author={{Menet}, G.},
			author={{Nordstr{\"o}m}, J.},
			author={S{\'a}~Earp, H.~N.},
			title={{Construction of ${G}_2$-instantons via twisted connected sums}},
			date={2015-10},
			journal={ArXiv e-prints},
			eprint={1510.03836},
		}
		
		\bib{Madsen2013}{article}{
			author={Madsen, T.~B.},
			author={Salamon, S.},
			title={Half-flat structures on ${S}^3\times {S}^3$},
			date={2013},
			ISSN={1572-9060},
			journal={Ann. Glob. Anal. Geom.},
			volume={44},
			number={4},
			pages={369\ndash 390},
			url={http://dx.doi.org/10.1007/s10455-013-9371-3},
		}
		
		\bib{Oliveira2014}{article}{
			author={Oliveira, G.},
			title={Monopoles on the {B}ryant--{S}alamon {$G\sb 2$}-manifolds},
			date={2014},
			ISSN={0393-0440},
			journal={J. Geom. Phys.},
			volume={86},
			pages={599\ndash 632},
			url={http://dx.doi.org/10.1016/j.geomphys.2014.10.005},
			review={\MR{3282350}},
		}
		
		\bib{Oliveira2014Thesis}{article}{
			author={Oliveira, Goncalo},
			title={Monopoles in higher dimensions},
			date={2014},
		}
		
		\bib{Oli16}{article}{
			author={Oliveira, G.},
			title={Calabi--{Y}au monopoles for the {S}tenzel metric},
			date={2016},
			ISSN={0010-3616},
			journal={Comm. Math. Phys.},
			volume={341},
			number={2},
			pages={699\ndash 728},
			url={http://dx.doi.org/10.1007/s00220-015-2534-2},
			review={\MR{3440200}},
		}
		
		\bib{SaEarp2015}{article}{
			author={S{\'a}~Earp, H.~N.},
			author={Walpuski, T.},
			title={{$\rm {G}\sb 2$}-instantons over twisted connected sums},
			date={2015},
			ISSN={1465-3060},
			journal={Geom. Topol.},
			volume={19},
			number={3},
			pages={1263\ndash 1285},
			url={http://dx.doi.org/10.2140/gt.2015.19.1263},
			review={\MR{3352236}},
		}
		
		\bib{Tian2000}{article}{
			author={Tian, G.},
			title={Gauge theory and calibrated geometry. {I}},
			date={2000},
			ISSN={0003-486X},
			journal={Ann. of Math. (2)},
			volume={151},
			number={1},
			pages={193\ndash 268},
			url={http://dx.doi.org/10.2307/121116},
			review={\MR{MR1745014 (2000m:53074)}},
		}
		
		\bib{Tao2004}{article}{
			author={Tao, T.},
			author={Tian, G.},
			title={A singularity removal theorem for {Y}ang--{M}ills fields in
				higher dimensions},
			date={2004},
			ISSN={0894-0347},
			journal={J. Amer. Math. Soc.},
			volume={17},
			number={3},
			pages={557\ndash 593 (electronic)},
			url={http://dx.doi.org/10.1090/S0894-0347-04-00457-6},
			review={\MR{MR2053951 (2005f:58013)}},
		}
		
		\bib{TsaiWang}{article}{
			author={Tsai, Chung-Jun},
			author={Wang, Mu-Tao},
			title={Mean curvature flows in manifolds of special holonomy},
			date={2018},
			ISSN={0022-040X},
			journal={J. Differential Geom.},
			volume={108},
			number={3},
			pages={531\ndash 569},
			url={https://doi.org/10.4310/jdg/1519959625},
			review={\MR{3770850}},
		}
		
		\bib{Walpuski2011}{article}{
			author={Walpuski, T.},
			title={{$\rm G\sb 2$}-instantons on generalised {K}ummer constructions},
			date={2013},
			ISSN={1465-3060},
			journal={Geom. Topol.},
			volume={17},
			number={4},
			pages={2345\ndash 2388},
			url={http://dx.doi.org/10.2140/gt.2013.17.2345},
			review={\MR{3110581}},
		}
		
		\bib{Walpuski2015}{article}{
			author={Walpuski, T.},
			title={{${\rm G}_2$}-instantons over twisted connected sums: an
				example},
			date={2016},
			ISSN={1073-2780},
			journal={Math. Res. Lett.},
			volume={23},
			number={2},
			pages={529\ndash 544},
			url={http://dx.doi.org/10.4310/MRL.2016.v23.n2.a11},
			review={\MR{3512897}},
		}
		
		\bib{Walpuski2017}{article}{
			author={Walpuski, T.},
			title={{${\rm G}_2$--instantons, associative submanifolds and Fueter
					sections}},
			date={2017},
			journal={Comm. Anal. Geom.},
			volume={25},
			number={4},
			pages={847\ndash 893},
		}
		
		\bib{Wang1958}{article}{
			author={Wang, H.-C.},
			title={On invariant connections over a principal fibre bundle},
			date={1958},
			ISSN={0027-7630},
			journal={Nagoya Math. J.},
			volume={13},
			pages={1\ndash 19},
			review={\MR{0107276}},
		}
		
	\end{biblist}
\end{bibdiv}
\end{document}